\documentclass[10pt,a4paper]{article}
\usepackage{amsfonts,amssymb,amsmath, amsthm}
\usepackage[draft]{fixme}
\theoremstyle{plain}

\newtheorem{theorem}{Theorem}[section]
\newtheorem{corollary}[theorem]{Corollary}
\newtheorem{lemma}[theorem]{Lemma}
\newtheorem{proposition}[theorem]{Proposition}

\theoremstyle{definition}
\newtheorem{definition}[theorem]{Definition}

\newtheorem{preremark}[theorem]{Remark}
\newenvironment{remark}{\begin{preremark}\normalfont}{\end{preremark}}

\allowdisplaybreaks

\newcommand{\e}{\mathcal E}
\newcommand{\F}{\mathcal F}

\newcommand{\R}{\mathbb R}

\DeclareMathOperator*{\esssup}{ess\,sup}
\DeclareMathOperator*{\essinf}{ess\,inf}

\begin{document}
\title{Scale-invariant boundary Harnack principle on inner uniform domains in fractal-type spaces}
\author{Janna Lierl}

\maketitle

\begin{abstract}
We prove a scale-invariant boundary Harnack principle for inner uniform domains in metric measure Dirichlet spaces. We assume that the Dirichlet form is symmetric, strongly local, regular, and that the volume doubling property and two-sided sub-Gaussian heat kernel bounds are satisfied. We make no assumptions on the pseudo-metric induced by the Dirichlet form, hence the underlying space can be a fractal space.
\end{abstract}

{\bf Acknowledgement:} Research partially supported by NSF grant DMS 1004771.
{\bf Address:} Malott Hall, Department of Mathematics, Cornell University, Ithaca, NY 14853, United States.

\section*{Introduction}

The boundary Harnack principle is a property of a domain that allows us to compare the decay of different non-negative harmonic functions on the domain under Dirichlet boundary condition. More precisely, a domain is said to satisfy a boundary Harnack principle if the ratio of any two non-negative harmonic functions is bounded near some part of the boundary of the domain where both functions vanish.

The boundary Harnack principle was introduced by Kemper \cite{Kem72}, followed by works of Dahlberg (\cite{Dah77}), Wu (\cite{Wu78}) and Ancona \cite{Anc78}. Whether a given domain has this property depends on the geometry of its boundary as well as on the precise formulation of the boundary Harnack principle.

The classical formulation is the \emph{global} boundary Harnack principle. It applies to functions $u,v$ that are positive harmonic on a domain $D$ and vanish continuously at the regular points of $(\partial D) \cap V$ and are bounded in a neighborhood of $(\partial D) \cap V$, where $V$ is an open set. Bass and Burdzy (\cite{BaBu91}) proved that, if $D$ is a twisted H\"older domain of order $\alpha \in (1/2,1]$, then the boundary Harnack inequality $\frac{u}{v} \leq A_1=A_1(D,V,K)$ is satisfied on $D \subset K$, for any compact set $K \subset V$.

In this paper, we are concerned with the \emph{scale-invariant} (also called \emph{geometric}) boundary Harnack principle. That is, $V$ and $K$ are replaced by two concentric balls of radius $A_0r$ and $r$, respectively, centered at a boundary point, and the constant $A_1=A_1(A_0)$ is independent of the radius $r$.

The scale-invariant boundary Harnack principle is crucial in order to identify the Martin boundary of a bounded domain as its topological boundary. Another interesting application is two-sided estimates of the Dirichlet heat kernel as in \cite{GyryaSC, LierlSC3}. In fact, the requirements on the domain in those works are mostly due to the need for a geometric boundary Harnack principle. 

The geometric boundary Harnack principle characterizes (inner) uniform domains, as was proved by Aikawa in \cite[Theorem 1.3 and Theorem 1.4]{Aik04}.

This determines a very large class of domains that are defined using the Euclidean metric (for uniform domains in $\R^n$) or, respectively, the inner metric of the domain (for inner uniform domains). The boundary of an (inner) uniform domain can be very rough (e.g.~the Koch snowflake). However, the domain, or any part of it, must not have the shape of a cusp. This condition affects the geometry of the domain both locally and globally and is needed to ensure the scale-invariance of the statement. A local boundary Harnack principle (i.e.~for small radius only) can be proved on domains that satisfy an inner uniform condition only locally.

On uniform domains in Euclidean space, the scale-invariant boundary Harnack principle was first proved by Aikawa (\cite{Aik01}). The result was extended to uniformly John domains in Euclidean space by Aikawa, Lundh, Mizutani (\cite[Theorem 3.1]{ALM03}). A "uniformly John domain" is the same as an "inner uniform domain". A different approach to the proof of the scale-invariant boundary Harnack principle on inner uniform domains is given in \cite[Corollaire 6.13]{Anc07}, along with related results for John domains. For more results on the boundary Harnack principle see also \cite{BV96a,BV96b}.

Gyrya and Saloff-Coste (\cite{GyryaSC}) generalized Aikawa's approach to uniform domains in (non-fractal) Dirichlet spaces. They considered a metric measure space equipped with a symmetric strongly local regular Dirichlet form that satisfies a parabolic Harnack inequality.  
Moreover, they deduced that the boundary Harnack principle holds on inner uniform domains, by considering the inner uniform domain as a uniform domain in a different metric 
space, namely the completion of the inner uniform domain with respect to its inner metric. 

A direct proof of the geometric boundary Harnack principle directly on inner uniform domains is given in (\cite{LierlSC1}) by Saloff-Coste and the author, following \cite{ALM03}. In \cite{LierlSC1}, the underlying metric measure space is equipped with a local, regular (possibly non-symmetric) Dirichlet form satisfying some technical assumptions (cf.~\cite{LierlSC2}) which provide control over the non-symmetric Dirichlet form in terms of its symmetric strongly local part. Furthermore, the symmetric stongly local part is assumed to satisfy the volume doubling property and the Poincar\'e inequality.

Both in \cite{GyryaSC} and \cite{LierlSC1}, the case of fractal spaces was excluded by putting certain hypotheses on the pseudo-metric induced by the Dirichlet form and by
assuming the classical space-time scaling $\Psi(r)=r^2$.

The aim of this paper is to give a proof of the geometric boundary Harnack principle in a more general setting that includes fractal spaces.
Applications of this result to proving two-sided bounds for the Dirichlet heat kernel on inner uniform domains will be presented in forthcoming papers \cite{KL1, KL2}.

In this paper, the underlying space $(X,d,\mu)$ is a metric measure space equipped with a symmetric strongly local regular Dirichlet form $(\e,\F)$ that satisfies volume doubling and weak two-sided heat kernel estimates w-HKE($\Psi$). The space-time scaling is captured by a continuous strictly 
increasing bijection $\Psi:[0,\infty) \to [0,\infty)$ which has the property that, for some constants $1 < \beta_1 \leq \beta_2 < \infty$ and $C_{\Psi} \in [1,\infty)$,
\begin{align} \label{eq:Psi beta}
 C_{\Psi}^{-1} \left( \frac{R}{r} \right)^{\beta_1} \leq \frac{\Psi(R)}{\Psi(r)} \leq C_{\Psi} \left( \frac{R}{r} \right)^{\beta_2}
 \quad \forall 0 < r \leq R < \infty.
 \end{align}

The main result of this paper is the following scale-invariant boundary Harnack principle.
\begin{theorem}
Let $(X,d,\mu,\e,\F)$ be a metric measure Dirichlet space that satisfies volume doubling \emph{(VD)}, and weak two-sided heat kernel estimates \emph{w-HKE($\Psi$)}. Assume that $(X,d)$ is complete and geodesic, and let $\Omega$ be an inner uniform domain in $(X,d)$.

Then there exist constants $A_0, A_1 \in (1,\infty)$ such that for any 
$\xi \in \widetilde\Omega \setminus \Omega$, $0 < r < R(\Omega)$,
and any two non-negative harmonic functions $u$ and $v$ on $B_{\Omega}(\xi,A_0r)$ with weak Dirichlet boundary condition along 
$\widetilde\Omega \setminus \Omega$, we have 
 \[  \frac{u(x)}{u(x')} \leq A_1 \frac{v(x)}{v(x')}, \quad \forall x, x' \in B_{\Omega}(\xi,r). \]
Here $B_{\Omega}(\xi,r) = \{ z \in \Omega : d_{\Omega}(\xi,z) < r \}$, $d_{\Omega}$ is the inner distance of the domain $\Omega$, and $\widetilde\Omega$ is the completion of $\Omega$ with respect to $d_{\Omega}$.
$R(\Omega)$ depends on the inner uniformity constants and on the diameter of $\Omega$, and can be chosen to be infinity if $\Omega$ is unbounded.
\end{theorem}

For example, removing the bottom line in the Sierpinski gasket produces a subset that is an inner uniform domain in the Sierpinski gasket. Hence, any two harmonic functions
on the Sierpinski gasket that vanish along the bottom line, decay at comparable rates near the bottom line.

The structure of this paper follows the presentation of the non-fractal case in \cite{LierlSC1}. We begin by recalling well-known definitions and properties of Dirichlet spaces  in Section \ref{sec:lw solutions}. In Section \ref{sec:PHI} we state the geometric hypotheses on the space, the volume doubling property and weak two-sided heat kernel estimates, and relate them to further conditions, EHI, RES($\Psi$), CSA($\Psi$), which will be used throughout the paper.
In Section \ref{sec:GF} we discuss the (local) inner uniformity conditions and collect some properties. In addition, we present estimates for the Dirichlet heat kernel 
on metric balls and for Green's functions on balls intersected with an inner uniform domain. Furthermore, we give detailed proofs of a maximum principle and a representation formula for harmonic functions, which were omitted in \cite{LierlSC1}. The latter results are essential ingredients in the proof of the geometric boundary Harnack principle (Theorem \ref{thm:bHP for u}) in Section \ref{sec:BHP}. Finally, in Section \ref{sec:applications}, we give two applications: First, the identification of the Martin boundary of a bounded inner uniform domain. Second, the construction of a harmonic profile on an unbounded inner uniform domain, which will be important to obtain bounds on the Dirichlet heat kernel in the forthcoming papers \cite{KL1, KL2}.

\section{Dirichlet space and harmonic functions} \label{sec:lw solutions}

\subsection{Dirichlet space and harmonic functions} \label{ssec:weak solutions}

Let $(X,d,\mu, \e, \F)$ be a metric measure Dirichlet space. That is, $X$ is a locally compact, separable, complete metric space, and all metric balls are relatively compact. $\mu$ is a positive Radon measure on $X$ with full support. 
$(\e, \F)$ is a symmetric, strongly local, regular Dirichlet form on $L^2(X,\mu)$, see \cite{FOT94}. 
We denote by $(L,D(L))$ the infinitesimal generator of $(\e,\F)$. 

There exists a measure-valued quadratic form $\Gamma$ defined by
 \[ \int_X f \, d\Gamma(u,u) = \e(uf,u) - \frac{1}{2} \e(f,u^2), \quad \forall f,u \in \F \cap L^{\infty}(X), \]
and extended to unbounded functions by setting $\Gamma(u,u) = \lim_{n \to \infty} \Gamma(u_n,u_n)$, where $u_n = \max\{\min\{u,n\},-n\}$. Using polarization, we obtain
a bilinear form $\Gamma$. In particular,
\[ \e(u,v) = \int_X d\Gamma(u,v), \quad \forall u,v \in \F. \]
The quadratic form $\Gamma(u,u)$ is called the \emph{energy measure} of the form $(\e,\F)$. It is denoted as $\mu_{\langle u \rangle}$ in \cite{FOT94}.
Further, $\Gamma$ satisfies 
\begin{align} \label{eq:Gamma(fg)}
 \int_X d\Gamma(fg,fg)
& \leq  2\int_X f^2 d\Gamma(g,g)  
      + 2\int_X g^2 d\Gamma(f,f), \quad f,g \in \F \cap L^{\infty}(X).
\end{align}
Here, on the right hand side, quasi-continuous versions of $f$ and $g$ must be used. See \cite[Lemma 3.2.5 and Lemma 5.6.1]{FOT94}

For $U \subset X$ open, set
 \[ \F_{\mbox{\tiny{{loc}}}}(U)  =  \left\{ f \in L^2_{\mbox{\tiny{loc}}}(U) : \forall \mbox{ open rel.~compact } A \subset U, \ \exists f^{\sharp} \in \F, f\big|_A = f^{\sharp}\big|_A \mbox{$\mu$-a.e.} \right\},   \]
where $L^2_{\mbox{\tiny{loc}}}(U)$ is the space of functions that are locally in $L^2(U)$.
For $f,g \in \F_{\mbox{\tiny{loc}}}(U)$ we can define $\Gamma(f,g)$ locally by $\Gamma(f,g)\big|_A = \Gamma(f^{\sharp},g^{\sharp})\big|_A$, where $A \subset U$ is open relatively compact and $f^{\sharp},g^{\sharp}$ are functions in $\F$ such that $f = f^{\sharp}$, $g = g^{\sharp}$ $\mu$-a.e.~on $A$. 

Define
\begin{align*}
 \F(U)   &=  \left\{ u \in \F_{\mbox{\tiny{loc}}}(U) : \int_U |u|^2 d\mu + \int_U d\Gamma(u,u) < \infty \right\}, \\
\F_{\mbox{\tiny{c}}}(U)  &=  \big\{ u \in \F(U) : \mbox{ The essential support of } u \mbox{ is compact in } U \big\}, \\
\F^0(U)  &= \mbox{ the closure of } \F_{\mbox{\tiny{c}}}(U) \mbox{ in } \F \mbox{ for the norm }
\left(\e(u,u) + \int_X u^2 d\mu \right)^{1/2}. 
\end{align*}
Note that $\F_{\mbox{\tiny{c}}}(U)$ is a linear subspace of $\F$.

\begin{definition}
Let $V \subset X$ be open. A function $u: V \to \R$ is \emph{harmonic} on $V$, if 
\begin{enumerate}
\item
$u \in \F_{\mbox{\tiny{loc}}}(V)$,
\item 
For any function $\phi \in \F_{\mbox{\tiny{c}}}(V), \ \e(u,\phi) = 0$.
\end{enumerate} 
If $u \in \F_{\mbox{\tiny{loc}}}(V)$ satisfies
$\e(u,\phi) \geq 0$ for all $\phi \in \F_{\mbox{\tiny{c}}}(V)$ with $\phi\geq 0$, then $u$ is called \emph{superharmonic}. 
\end{definition}

\subsection{Dirichlet-type Dirichlet form}

\begin{definition}
For an open set $U \subset X$, the Dirichlet-type Dirichlet form on $U$ is defined as
\[ \e^D_U(f,g) := \e(f,g),  \quad f,g \in \F^0(U). \]
\end{definition}
Let $(L^D_U,D(L^D_U))$ be the infinitesimal generator and $P^D_{U,t}$, $t > 0$, be the semigroup associated with $(\e^D_U,\F^0(U))$. If the semigroup admits an integral kernel, that is, a non-negative Borel measurable map $p^D_U(t,x,y):(0,\infty) \times U \times U \to \R$ such that, for any $f \in L^2(U,\mu)$,
\[ P^D_{U,t} f (x) = \int_U p^D_U(t,x,y) f(y) \mu(dy), \quad \forall t>0, x \in U,  \]
then $p^D_U(t,x,y)$ is called the \emph{Dirichlet heat kernel} on $U$. Using the reasoning in \cite[Proposition 2.3]{SturmII}, one can show that the map 
$y \mapsto p^D_U(t,x,y)$ is in $\F^0(U)$.
$p^D_X(t,x,y)$ is simply denoted as $p(t,x,y)$ and called the \emph{heat kernel} on $X$.

The extended Dirichlet space $\F^0(U)_e$ is defined as the family of all measurable, almost everywhere finite functions $u$ such that there exists an approximating sequence $(u_n) \subset \F^0(U)$ that is $\e^D_U$-Cauchy and $u = \lim u_n$ $\mu$-almost everywhere. If $(\e^D_U,\F^0(U))$ is transient then $\F^0(U)_e$ is complete, by \cite[Lemma 1.5.5]{FOT94}.

$\e^D_U$ is canonically extended to a bilinear form on $\F^0(U)_e$ by setting $\e(u,u) := \lim_{n \to \infty} \e(u_n,u_n)$, where $u_n \in \F^0(U)$ is any approximating sequence for $u \in \F^0(U)_e$.

\subsection{Capacity}
The potential theory for symmetric regular Dirichlet forms is developed in \cite[Chapter 2]{FOT94}. Here we recall some definitions and facts that we need in the sequel.

Let $U \subset X$ be open. Suppose $(\e^D_U,\F^0(U))$ is transient. 
For an open subset $A \subset U$ let
 \[ \mathcal{L}_{A,U} = \{ u \in \F^0(U)_e : u \geq 1 \mbox{ $\mu$-a.e. on } A \}. \] 
The \emph{$0$-capacity} of $A$ in $U$ is defined as
 \[ \mbox{Cap}_{U} (A) = \begin{cases} \inf_{u \in \mathcal{L}_{A,U}} \e^D_U(u,u), & \mathcal{L}_{A,U} \neq \emptyset \\
                                                   + \infty,      & \mathcal{L}_{A,U} = \emptyset. 
                                     \end{cases}
\]
For an arbitrary subset $A$ of $U$, the capacity is defined as
\[ \mbox{Cap}_{U} (A) := \inf \{ \mbox{Cap}_{U}(B) : B \mbox{ is an open subset of $U$ with } A \subset B \}. \]
Notice the set monotonicity of the capacity: If $A \subset B$ are subsets of $U$, then $\mbox{Cap}_{U} (A) \leq \mbox{Cap}_{U} (B)$. If
$V \subset U$ is open, then $\mbox{Cap}_{U} (A) \leq \mbox{Cap}_{V} (A)$.
Now for an arbitrary subset $A \subset U$, let
 \[ \mathcal{\widetilde L}_{A,U} = \{ u \in \F^0(U)_e : \widetilde u \geq 1 \mbox{ q.e. on } A \}, \] 
where $\widetilde u$ is a quasi-continuous modification of $u$.
If $\mathcal{\widetilde L}_{A,U} \neq \emptyset$ then there exists a unique element $e_{A} \in \mathcal{L}_{A,U}$ such that
 \[ {\mbox{Cap}}_{U} (A) = \e^D_U(e_A,e_A). \]
Moreover, $e_A$ is the unique element of $\F^0(U)_e$ satisfying $\widetilde e_A=1$ q.e.~on $A$ and $\e^D_U(e_{A},v) \geq 0$ for all $v \in \F^0(U)_e$ with $\widetilde v \geq 0$ q.e.~on $A$. See \cite[Theorem 2.1.5 and p.71]{FOT94}.
 
Now assume that $A$ is closed.
By the $0$-order counterpart of \cite[Lemma 2.2.6]{FOT94} (cf.~also \cite[Lemma 2.2.10]{FOT94}), $e_{A}= \mathcal{G}_U \nu_A$ is a ($0$-order) potential. That is, the \emph{equilibrium measure} $\nu_A$ is a Radon measure with $\nu_A(U\setminus A) = 0$ which charges no set of zero capacity, and it holds
\[ \e^D_U(\mathcal{G}_U \nu_A,v) = \int v \, d\nu_A, \qquad \forall v \in \F \cap C_0(U), \]
where $C_0(X)$ is the space of real continuous functions with compact support in $X$.
In particular, we have
\begin{align} \label{eq:Cap nu}
  {\mbox{Cap}}_{U} (A)
 =  \e^D_U(e_{A},e_{A}) 
 =  \e^D_U( \mathcal{G}_U \nu_A,e_{A})
  = \int \widetilde{e_{A}} \, d\nu_A
 = \nu_A(A).
\end{align}


\section{Geometric hypotheses on the underlying space} \label{sec:PHI}

Let $(X,d,\mu, \e, \F)$ be a MMD space. In this section we describe the geometric hypotheses on the underlying space $X$. For examples of fractal-type spaces that satisfy these conditions, see e.g. \cite{BP88, BaBa92, Kum93, FHK94, Baba99, BBK06}.

\begin{definition}
$(X,\mu)$ satisfies the \emph{volume doubling property} VD if there exists a constant $C_{\mbox{\tiny{VD}}} \in (0,\infty)$ such that for every $x \in X$, $R>0$,
 \[  V(x,2R) \leq C_{\mbox{\tiny{VD}}} \, V(x,R),  \]
where $V(x,R) = \mu( B(x,R))$ denotes the volume of the ball $B(x,R)=\{ y \in X : d(x,y) < R \}$.
\end{definition} 
 
Let $\Psi:[0,\infty) \to [0,\infty)$ be a continuous strictly increasing  bijection satisfying \eqref{eq:Psi beta} for some constants $1 < \beta_1 \leq \beta_2 < \infty$ and $C_{\Psi} \in [1,\infty)$.


\begin{definition}
\begin{enumerate}
\item
$(\e,\F)$ satisfies the \emph{weak two-sided heat kernel bounds} w-HKE($\Psi$) if the heat
kernel $p(t,x,y)$ on $X$ exists and if there are positive constants $c_1$, $c_2$, $c_3$, $c_4$ and $\epsilon > 0$ so that
\begin{align*}
 p(t,x,y) \leq    \frac{c_1}{ V(x,\Psi^{-1}(t)) } \exp \left( - c_2 \left( \frac{\Psi(d(x,y))}{t} \right)^{\frac{1}{\beta_2-1}}\right)
 \end{align*}
for all $t \in (0,\infty)$ and almost every $x,y \in X$, and
\begin{align*}
p(t,x,y)
\geq \frac{c_3}{ V(x,\Psi^{-1}(t)) }
\end{align*}
for all $t \in (0,\infty)$ and almost every $x,y \in X$ with $d(x,y) \leq \epsilon \Psi^{-1}(t)$.
\item
$(\e,\F)$ satisfies the \emph{weak local lower estimate} w-LLE($\Psi$) if there exist $c_5>0$ and $\epsilon \in (0,1)$ such that, for any ball $B=B(x_0,R) \subset X$, the Dirichlet heat $p^D_B(t,x,y)$ exists and satisfies
\begin{align} \label{eq:LLE}
p^D_B(t,x,y) \geq \frac{c_5}{ V(x_0,\Psi^{-1}(t)) }
\end{align} 
for all $0 < t \leq \Psi(R)$, and almost every $x,y \in B(x_0,\epsilon \Psi^{-1}(t))$.
\end{enumerate}
\end{definition}

\begin{remark} \label{rem:LLE}
In \cite{BGK12}, a weaker version of w-LLE($\Psi$) is defined: \eqref{eq:LLE} is required to hold only for $t \leq \Psi(\epsilon R)$. Assuming VD, their w-LLE($\Psi$) is proved to be equivalent to w-HKE($\Psi$) in \cite[Theorem 3.1]{BGK12}. It is not hard to modify the proof to show the equivalence of w-HKE($\Psi$) with our  version of w-LLE($\Psi$).
\end{remark}

\begin{theorem} \label{thm:w-HKE w-PHI} 
Assume \emph{VD} is satisfied. Then the following are equivalent:
\begin{enumerate}
\item
$X$ satisfies \emph{w-HKE($\Psi$)}.
\item
$X$ satisfies \emph{w-LLE($\Psi$)}.
\end{enumerate}
Further, under any of the conditions (i), (ii), the heat kernel $p(t,x,y)$ is a continuous function of $(t,x,y) \in (0,\infty) \times X \times X$, and, for any open subset $U \subset X$, the Dirichlet heat kernel $p^D_{U}(t,x,y)$ is a continuous function of $(t,x,y) \in (0,\infty) \times U \times U$.
\end{theorem}

\begin{proof}
See \cite[Theorem 3.1]{BGK12} and Remark \ref{rem:LLE}.
\end{proof}

\begin{remark}
Assuming VD, condition w-HKE($\Psi$) is also equivalent to a weak parabolic Harnack inequality w-PHI($\Psi$), see \cite[Theorem 3.1]{BGK12}.
\end{remark}

%

\begin{definition}
Let $V,U$ be open sets in $X$ with $V \subset \overline{V} \subset U$. A function
$\phi \in \F$ is called a cutoff function for $V$ in $U$ if $0 \leq \phi \leq 1$ $\mu$-a.e.~on $X$, $\phi = 1$ $\mu$-a.e.~on $V$ and $\widetilde{\phi} = 0$ q.e.~on $X \setminus U$.
\end{definition}

The following variants CSD and CSA of the cutoff Sobolev inequality have been introduced in \cite{AB13}. Note that in this paper we do not require cutoff functions to be continuous.

\begin{definition}
Let $D_0$, $D_1$ be open subsets of $X$ with $D_0 \subset \overline{D_0} \subset D_1$, and let $U =
D_1 \setminus D_0$. We say that property CSD($D_0$,$D_1$,$\theta$) holds if there exists a cutoff function $\phi$ for $D_0 \subset D_1$ which satisfies, for all $f \in \F$,
 \begin{align} \label{eq:CSD}
 \int_U f^2 d\Gamma(\phi,\phi) \leq \frac{1}{8} \int_U \phi^2 d\Gamma(f,f) + \theta \int_U f^2 d\mu.
 \end{align}
\end{definition}

\begin{definition}
We say that condition CSA($\Psi$) holds if there exists a constant $C_S>0$ such that for every $x \in X$, $R>0$, $r>0$ the condition CSD($B(x,R)$, $B(x,R+r)$, $C_S/\Psi(r)$) holds.
\end{definition}

\begin{definition}
$(\e,\F)$ satisfies the \emph{elliptic Harnack inequality} EHI, if there exist constants $C>0$ and $\delta \in (0,1)$ such that, for any ball $B(x,R) \subset X$ and any non-negative function $u \in \F$ that is harmonic on $B(x,R)$, we have
\[ \esssup_{z \in B(x,\delta R)} u(z) \leq C \essinf_{z \in B(x,\delta R)} u(z). \]
\end{definition}

\begin{definition}
$(\e,\F)$ satisfies the \emph{weak Poincar\'e inequality} PI($\Psi$) on $X$ if there exist constants $C_{\mbox{\tiny{PI}}} \in (0,\infty)$ and $\kappa \geq 1$ such that for any ball $B = B(x,R) \subset X$ and any $f \in \F$,
\[ \int_B (f - f_B)^2 d\mu \leq C_{\mbox{\tiny{PI}}} \Psi(R) \int_{B(x,\kappa R)} d\Gamma(f,f).  \]
Here $f_B = \mu(B)^{-1} \int_B f d\mu$.
\end{definition}

\begin{definition}
$(\e,\F)$ satisfies the \emph{resistance condition} RES($\Psi,K$) if there exist constants $K, C > 1$ such that for any $B(x,(K+1)R) \subsetneq X$,
 \[ C^{-1} \frac{\Psi(R)}{V(x,R)} \leq \left(\mbox{Cap}_{B(x,KR)}(B(x,R)) \right)^{-1} \leq C \frac{\Psi(R)}{V(x,R)}.   \]
 We say that RES($\Psi$) is satisfied if RES($\Psi,K$) is satisfied for all $K>1$.
\end{definition}

\begin{theorem} \label{thm:VD+PI+CSA}
Assume that $(X,d,\mu,\e,\F)$ satisfies \emph{VD} and \emph{w-HKE($\Psi$)}. Then $(\e,\F)$ is conservative and satisfies \emph{EHI}, \emph{CSA($\Psi$)}, \emph{PI($\Psi$)}, and \emph{RES($\Psi$)}. Moreover, $X$ is connected.
\end{theorem}

The classical proof of the Poincar\'e inequality (cf.~\cite{SC92}) uses the semigroup associated with the Neumann-type form 
\[ \e^N_U(f,g) := \int_U d\Gamma(f,g),  \quad f,g \in \F(U), \]
where $U \subsetneq X$ is an open set so that $(\e^D_U,\F^0(U))$ is transient.
A proof in \cite{BBK06} claimed that the Neumann semigroup admits a kernel; it is not clear to the author that this is the case. Here we give a proof that does not rely on the existence of the Neumann kernel. 

\begin{proposition}
Assume that \emph{CSA($\Psi$)} holds. Then the Neumann-type form $(\e^N_U,\F(U))$ is a strongly local Dirichlet form on $L^{2}(U, \mu)$.
\end{proposition}

\begin{proof}
It is clear from the definition that $(\e^N_U,\F(U))$ is a strongly local Markovian symmetric form. The closedness follows as in the proof of \cite[Proposition 2.51]{GyryaSC}, and by applying CSA($\Psi$).
\end{proof}

Let $P^N_{U,t}$, $t>0$, be the strongly continuous contraction semigroup associated with the Neumann-type form $(\e^N_U,\F(U))$.

\begin{lemma} \label{lem:N D}
Let $B=B(x,R) \subset X$. Let $\epsilon \in (0,1)$ and $\epsilon B = B(x,\epsilon R)$. Then the Dirichlet semigroup on $\epsilon B$ is dominated by the Neumann semigroup on $B$. That is,
\begin{align} \label{eq:P^N>P^D}
P^N_{B,t} f \geq P^D_{\epsilon B,t} f   \quad \forall 0 \leq f \in L^2(\epsilon B), t>0.
\end{align} 
\end{lemma}

\begin{proof}
Let $f \in L^2(\epsilon B)$, $f \geq 0$. Then
\[ \left(P^D_{\epsilon B,t} f - P^N_{B,t} f \right)^+ \leq P^D_{\epsilon B,t} f. \]
Hence, by \cite[Lemma 4.4]{GH08}, $\left(P^D_{\epsilon B,t} f - P^N_{B,t} f \right)^+ \in \F(B)$. Since $P^D_{\epsilon B,t} f$ and hence $\left(P^D_{\epsilon B,t} f - P^N_{B,t} f \right)^+$ can be approximated by functions in $\F_{\mbox{\tiny{c}}}(\epsilon B)$, we find that $\left(P^D_{\epsilon B,t} f - P^N_{B,t} f \right)^+ \in \F^0(\epsilon B)$.
Now we follow \cite[Proof of Theorem 3.3]{Ou96}.
Let $u= P^D_{\epsilon B,t} f$ and $v = P^N_{B,t} f$. Then
\begin{align*}
 \e^D_{\epsilon B}(u,(u-v)^+) 
& = \e^N_{\epsilon B}(u-v,(u-v)^+) + \e^N_{\epsilon B}(v,(u-v)^+)  \\
& \geq \e^N_{\epsilon B}(v,(u-v)^+) = \e^N_{B}(v,(u-v)^+).
\end{align*}
Since $L^D_{\epsilon B} u = \frac{d}{dt} u$ and $L^N_{B} v = \frac{d}{dt} v$, we can rewrite the above inequality as
\[ \int \left(\frac{d}{dt} u \right) (u-v)^+ d\mu \leq \int \left( \frac{d}{dt} v \right)  (u-v)^+ d\mu. \]
Hence, the function $g(t) = \int ((u-v)^+)^2 d\mu$ satisfies $\frac{d}{dt}g \leq 0$ and $g(0)=0$. Therefore $(u-v)^+=0$, which proves the claim.
\end{proof}

\begin{proof}[Proof of Theorem \ref{thm:VD+PI+CSA}]
The weak heat kernel bounds w-HKE($\Psi$) are equivalent to a weak parabolic Harnack inequality w-PHI($\Psi$), by \cite[Theorem 3.1]{BGK12}.  w-PHI($\Psi$) implies the elliptic Harnack inequality for non-negative harmonic functions that are bounded.
Then, \cite[Proof of Theorem 7.4, (iii) $\Rightarrow$ (H)]{GT12} shows that EHI holds for general non-negative harmonic functions on a ball $B$ for which the smallest Dirichlet eigenvalue $\lambda_{\mbox{\tiny{min}}}(B)$ is positive. The assumption $\lambda_{\mbox{\tiny{min}}}(B)>0$ can be dropped: \cite[Theorem 4.6.5]{FOT94} implies that for any ball $B \subset X$ and $f \in \F$, there exists of a solution $u = H_B \widetilde f$ to the Dirichlet problem
\begin{align*}
\begin{cases}
& u \mbox{ harmonic on B }, \\
& f - u \in \F^0(B).
\end{cases} 
\end{align*}
As can be seen from \cite[Proof of Theorem 4.6.5]{FOT94}, the operator $H_B$ is bounded. Hence, we can avoid the use of \cite[Lemma 7.1]{GT12} and apply \cite[Lemma 7.2(c)]{GT12} directly. This proves EHI.
By \cite[Lemma 7.3(a)]{GT12}, it follows that $X$ is connected.

Notice that the weak heat kernel estimates w-HKE($\Psi$) are equivalent to the upper estimate $\mbox{UE}_{\mbox{\tiny weak}}$ together with the near diagonal lower bound $\mbox{NLE}_{\mbox{\tiny weak}}$ in the notation of \cite{GT12}. Indeed, by \cite[Lemma 3.19]{GT12} it follows that $\Phi(d(x,y),t) \geq \left(\frac{\Psi(d(x,y))}{t} \right)^{\frac{1}{\beta_2-1}}$, where $\Phi$ is as in \cite[(3.34)]{GT12}, hence $\mbox{UE}_{\mbox{\tiny weak}}$ and $\mbox{NLE}_{\mbox{\tiny weak}}$ imply w-HKE($\Psi$). Clearly, w-HKE($\Psi$) implies $\mbox{NLE}_{\mbox{\tiny weak}}$. $\mbox{UE}_{\mbox{\tiny weak}}$ follows from  w-PHI($\Psi$) by slightly modifying the argument given in \cite[Proof of Theorem 3.1]{BGK12}: It is immediate to extend \cite[(4.66)]{BGK12} to $k \in [1,\infty)$ by replacing the constant $c_2$ with $c_2^2$. Repeating the argument in \cite[Subsection 4.3.6]{BGK12} with general $k \in [1,\infty)$ yields \cite[(4.72)]{BGK12},
\begin{align*}
 p_t(x,y)
& \leq 
\frac{C c_2^k}{V(x,\Psi^{-1}(t))} \exp \left( - \frac{c_1 d(x,y)}{\Psi^{-1}(t/k)} \right) \\
& = 
\frac{C }{V(x,\Psi^{-1}(t))} \exp \left( k \log c_2 - \frac{c_1 d(x,y)}{\Psi^{-1}(t/k)} \right). 
\end{align*}
Optimizing over $k \in [1,\infty)$ and using \cite[(3.35)]{GT12} with $\lambda = \frac{k}{t}$ yields
\begin{align*}
 p_t(x,y)
\leq 
\frac{C }{V(x,\Psi^{-1}(t))} \exp \left( - \Phi(c \, d(x,y),t) \right), 
\end{align*}
which is the desired upper bound $\mbox{UE}_{\mbox{\tiny weak}}$.


Next, we show that $(\e,\F)$ is conservative. 
Then the implication VD, w-HKE($\Psi$) $\Rightarrow$  CSA($\Psi$) follows easily by adapting the arguments of \cite[Theorem 5.5, Lemma 5.3]{AB13} to the present setting. The proof of \cite[Lemma 5.3]{AB13} refers to \cite{GH10}, where conservativeness of the form is assumed.

In the case that $X$ is bounded, it is clear that $(\e,\F)$ is conservative because the constant function $1$ is in $\F$.

In the case when $X$ is unbounded, it is proved in \cite[Theorem 7.4]{GT12} that $\mbox{UE}_{\mbox{\tiny weak}}$ together with $\mbox{NLE}_{\mbox{\tiny weak}}$ are equivalent to EHI together with an exit time estimate E($\Psi$). 
By \cite[Lemma 7.3]{GT12}, condition E($\Psi$) implies that the form is conservative.

Next, we show PI($\Psi$). We follow the line of reasoning in the proof of \cite[Theorem 5.5.1]{SC02}.
Let $P^N_{B,t}$, $t>0$, be the Neumann semigroup on the ball $B=B(x,\frac{1}{\epsilon} R)$, and $P^D_{\epsilon B,t}$, $t>0$, the Dirichlet semigroup on the ball $\epsilon B=B(x,R)$, where $\epsilon  \in (0,1)$ is the parameter in w-LLE($\Psi$). Recall that w-HKE($\Psi$) is equivalent to w-LLE($\Psi$) by Theorem \ref{thm:w-HKE w-PHI}. By Lemma \ref{lem:N D},
\begin{align}
P^N_{B,t} f \geq P^D_{\epsilon B,t} f \quad \forall 0 \leq f \in L^2(\epsilon B), t>0.
\end{align} 
Let $f \in \F(B)$, $y \in B(x,\epsilon R)$. Let $\phi$ be a cutoff function that is $1$ on $B(x,\epsilon R)$ and has compact support in $B(x,R)$. Applying \eqref{eq:P^N>P^D} and w-LLE($\Psi$) with $t=\Psi(R)$, we obtain
\begin{align*}
& P^{B,N}_{\Psi(R)} \left([ f - P^N_{B,\Psi(R)} f(y) ]^2\right) (y) \\
& \geq P^D_{\epsilon B,\Psi(R)} \left(\phi[ f - P^N_{B,\Psi(R)} f(y) ]^2\right) (y) \\
& = \int_B p^D_{\epsilon B}(\Psi(R),y,z) \phi(z)|f(z)-P^N_{B,\Psi(R)} f(y)|^2 \mu(dz) \\
& \geq \int_{B(x,\epsilon R)} p^D_{\epsilon B}(\Psi(R),y,z) |f(z)-P^N_{B,\Psi(R)} f(y)|^2 \mu(dz) \\
& \geq \frac{c_5}{ V(x,R) } \int_{B(x,\epsilon R)}  |f(z)-P^N_{B,\Psi(R)} f(y)|^2 \mu(dz) \\
& \geq \frac{c_5}{ V(x,R) } \int_{B(x,\epsilon R)}  |f(z)-f_{B(x,\epsilon R)}|^2 \mu(dz),
\end{align*}
where $f_{B(x,\epsilon R)}$ is the mean of $f$ over the ball $B(x,\epsilon R)$.
Integrating over $B(x,\epsilon R)$ and using volume doubling yields
\begin{align} \label{eq:5.5.6}
& \int_{B(x,\epsilon R)} P^N_{B,\Psi(R)} \left([ f - P^N_{B,\Psi(R)} f(y) ]^2\right) (y) \mu(dy) \nonumber \\
& \geq c_5 \frac{V(x,\epsilon R)}{V(x,R)} 
        \int_{B(x,\epsilon R)}  |f(z)-f_{B(x,\epsilon R)}|^2 \mu(dz) \nonumber \\
& \geq c(\epsilon) \int_{B(x,\epsilon R)}  |f(z)-f_{B(x,\epsilon R)}|^2 \mu(dz).
\end{align}
The next computation uses the fact that the Neumann semigroup $P^N_{B,t}$, $t>0$, is conservative, i.e.~$P^N_{B,t} 1=1$ for all $t>0$, and an $L^2$-contraction.
\begin{align}  \label{eq:5.5.7}
& \int_{B(x,\epsilon R)} P^N_{B,\Psi(R)} \left([ f - P^N_{B,\Psi(R)} f(y) ]^2\right) (y) \mu(dy)  \nonumber \\
& = \int_{B(x,\epsilon R)} P^N_{B,\Psi(R)}(f^2) - 2 \left[ P^N_{B,\Psi(R)} f(y) \right]^2 + P^N_{B,\Psi(R)}1(y) \left[P^N_{B,\Psi(R)} f(y) \right]^2 \mu(dy) \nonumber \\
& \leq 
\Vert f \Vert^2_{B(x,\epsilon R),2} - \Vert P^N_{B,\Psi(R)}  f \Vert^2_{B(x,\epsilon R),2}
 = - \int_0^{\Psi(R)} \frac{\partial}{\partial s} \Vert P^N_{B,s}  f \Vert^2_{B(x,\epsilon R),2} ds \nonumber \\
& = 
- 2 \int_0^{\Psi(R)} \langle L^N_B P^N_{B,s} f, P^N_{B,s} f \rangle ds \nonumber 
 = 2 \int_0^{\Psi(R)} \e^N_B( P^N_{B,s} f, P^N_{B,s} f ) ds \nonumber \\
& \leq 
2 \Psi(R) \e^N_B(f,f) 
 = 2 \Psi(R) \int_B d\Gamma(f,f).
\end{align}
To see the last inequality, observe that $s \to \e^N_B( P^N_{B,s} f, P^N_{B,s} f )$ is a non-increasing function. This can be proved by noting that
\[ \e^N_B( P^N_{B,s} f, P^N_{B,s} f ) = - \langle L^N_B P^N_{B,s} f, P^N_{B,s} f \rangle = \Vert (L^N_B)^{1/2} P^N_{B,s} f \Vert^2_2. \]
From \eqref{eq:5.5.6} and \eqref{eq:5.5.7} we obtain
 \[  \int_{B(x,\epsilon R)}  |f(z)-f_{B(x,\epsilon R)}|^2 \mu(dz)
 \leq \frac{2}{c(\epsilon)} \Psi(R) \int_B d\Gamma(f,f). \]
We have proved that VD and w-HKE($\Psi$) imply PI($\Psi$).

Now RES($\Psi$) follows from VD and CSA($\Psi$) together with PI($\Psi$) by the same arguments as in \cite[Lemma 5.1]{BaBa04}.
\end{proof}

\section{Inner uniform domains and Green's function estimates} \label{sec:GF}

\subsection{(Inner) uniformity}  \label{ssec:uniform domains}
For the rest of this paper, we let $(X,d,\mu,\e, \F)$ be a MMD space that satisfies VD and w-HKE($\Psi$).
From now on, $(X,d)$ is always assumed to be geodesic. That is, for any $x,y \in X$ there is a continuous map $\gamma:[0,1] \to X$ such that $\gamma(0)=x$, $\gamma(1)=y$ and $d(\gamma(s),\gamma(t)) = |s-t| d(x,y)$ for all $s,t \in [0,1]$. 

In this Section we discuss inner uniformity conditions and some properties of inner uniform domains. Section \ref{ssec:uniform domains} is nearly identical to the corresponding section in \cite{LierlSC1}. We include it for the convenience of the reader and due to its importance for the understanding of the main result.

Let $\Omega \subset X$ be open and connected.
The \emph{inner metric} on $\Omega$ is defined as
 \[ d_{\Omega}(x,y) = \inf \big\{ \mbox{length}(\gamma) \big| \gamma:[0,1] \to \Omega \mbox{ continuous}, \gamma(0) = x, \gamma(1) = y \big\}, \]
where 
\[ \mbox{length}(\gamma) = \sup \left\{ \sum_{i=1}^n d(\gamma(t_i),\gamma(t_{i-1})) : n \in \mathbb{N}, 0 \leq t_0 < \ldots < t_n \leq 1 \right\}. \]
Let $\widetilde \Omega$ be the completion of $\Omega$ with respect to $d_{\Omega}$, and $\partial_{\widetilde \Omega} \Omega = \widetilde\Omega \setminus \Omega$ the boundary. 
For $x \in \widetilde\Omega$, we will consider the \emph{inner balls} $B_{\widetilde \Omega}(x,r) := \{ y \in \widetilde \Omega : d_{\Omega}(x,y)<r \}$ and $B_{\Omega}(x,r) := B_{\widetilde\Omega}(x,r) \cap \Omega$.

For an open set $B \subset \Omega$, let $\overline{B}^{d_{\Omega}}$ be its completion with respect to the metric $d_{\Omega}$, and let $\partial_{\widetilde \Omega} B = \overline{B}^{d_{\Omega}} \setminus B$ be its boundary. Let $\partial_{\Omega} B = \partial_{\widetilde \Omega} B \cap \Omega$ be the part of the boundary that lies in $\Omega$.
Observe that $\partial_{\Omega} B = \partial_X B \cap \Omega$.

For a set $V \subset \Omega$ let $\mbox{diam}_{\Omega}(V):= \sup_{x,y \in V} d_{\Omega}(x,y)$ be the diameter of $V$ with respect to the inner metric $d_{\Omega}$.

\begin{definition} \label{def:uniformity}
\begin{enumerate}
\item 
Let $\gamma: [\alpha,\beta] \to \Omega$ be a rectifiable curve in $\Omega$ and let $c \in (0,1)$, $C \in (1,\infty)$. We call $\gamma$ a $(c,C)$-\emph{uniform curve} in $\Omega$ if
\begin{equation} \label{eq:def uniform curve}
d \big( \gamma(t), \overline \Omega \setminus \Omega \big) \geq c \cdot \min \left\{ d \big( \gamma(\alpha), \gamma(t) \big) , d \big( \gamma(t), \gamma(\beta) \big) \right\},    \quad \forall t \in [\alpha,\beta],
\end{equation}
and if
\[ \mbox{length}(\gamma) \leq C \cdot d \big( \gamma(\alpha), \gamma(\beta) \big). \]
The domain $\Omega$ is called $(c,C)$-\emph{uniform} if any two points in $\Omega$ can be joined by a $(c,C)$-uniform curve in $\Omega$.
\item
\emph{Inner uniformity} is defined analogously by replacing the metric $d$ on $X$ with the inner metric $d_{\Omega}$ on $\Omega$. 
\item
The notion of \emph{(inner) $(c,C)$-length-uniformity} is defined analogously by replacing the right hand side of \eqref{eq:def uniform curve} by $c \cdot \min\{\mbox{length} (\gamma\big|_{[a,t]}),\mbox{length} (\gamma\big|_{[t,b]}) \}$.
\end{enumerate}
\end{definition}

The following proposition is an easy consequence of \cite[Proposition 3.3 and Theorem 3.7]{GyryaSC}. See also \cite[Lemma 2.7]{MS79}.

\begin{proposition} 
Assume $(X,d,\mu)$ satisfies VD. Then a connected open subset $\Omega \subset X$ is (inner) uniform if and only if it is (inner) length-uniform.
\end{proposition} 

\begin{lemma} \label{lem:x_r} 
Let $\Omega$ be a $(c_u,C_u)$-inner uniform domain in $(X,d)$.
For every inner ball $B = B_{\widetilde \Omega}(x,r)$ with minimal radius (i.e.~$B \neq B_{\widetilde \Omega}(x,R)$ for all $R >r$), there exists a point $x_r \in B$ with $d_{\Omega}(x,x_r) = r/4$ and $d_{\Omega}(x_r,\widetilde \Omega \setminus \Omega) \geq c_ur/8$.
\end{lemma}

\begin{proof} See \cite[Lemma 3.20]{GyryaSC}.
\end{proof}

The next lemma is crucial for the proof of the boundary Harnack principle on inner uniform domains rather than uniform domains.

Let $p:\widetilde\Omega \to \overline{\Omega}$ be the natural projection into the closure $\overline{\Omega}$ of $\Omega$ in $(X,d)$, namely the unique continuous map such that $p|_{\Omega}$ is the identity map on $\Omega$. For any $x \in \widetilde {\Omega}$ and any ball $D=B(p(x),r)$, let $D'$ be the unique connected component of $p^{-1}(D \cap \overline{\Omega})$ that contains $x$. 
We identify the subset $D'\cap p^{-1}(\Omega)$ of  $(\widetilde \Omega,d_{\Omega})$ with the subset $p(D') \cap \Omega $ of $ (X,d)$ and simply denote it by $D' \cap \Omega$.
 It follows that $D'\cap \Omega$ is the connected component of $D\cap \Omega$ whose closure in $\widetilde{\Omega}$ contains $x$.

\begin{lemma} \label{lem:metrics are comparable}
Suppose $\mu$ has the volume doubling property on $X$. 
Let $\Omega$ be a $(c_u,C_u)$-inner uniform domain in $(X,d)$.
Then there exists a constant $C_{\Omega} \in (0,\infty)$ such that for any ball $D = B(p(x),r/C_{\Omega})$ with $x \in \widetilde\Omega$,
 \[ B_{\widetilde\Omega}(x,r/C_{\Omega}) \subset D' \subset B_{\widetilde\Omega}(x, r). \]
The constant $C_{\Omega}$ depends only on the volume doubling constant $C_{\mbox{\em \tiny{VD}}}$ and on the inner uniformity constants $c_u$, $C_u$ of $\Omega$.
\end{lemma}

\begin{proof}
The proof is given in \cite[Lemma 3.7]{LierlSC1} and follows earlier results for domains in Euclidean space, \cite[Lemma 2.2]{ALM03} and \cite[(5.5), (5.6) and the comment thereafter]{Anc07}.
\end{proof}

\begin{remark}
In the two above lemmas, the hypothesis that the whole domain $\Omega$ is $(c_u,C_u)$-inner uniform can be relaxed to the hypothesis that any two points in $B_{\Omega}(x, r)$ can be connected by a path that is $(c_u,C_u)$-inner uniform in $\Omega$. See \cite[Remark 3.8(ii)]{LierlSC1}.
\end{remark}

Fix $c_u \in (0,1)$ and $C_u \in (1,\infty)$. Let $A_3 = 2(12 + 12C_u)$, $A_0 = A_3 + 7$. We will use the notation $a \vee b = \max\{a,b\}$ and $a \wedge b = \min\{a,b\}$ for $a,b \in \R$.

\begin{definition} \label{def:R_xi}
For $\xi \in \partial_{\widetilde{\Omega}} \Omega$, let $R_{\xi}$ be the largest radius so that 
\begin{enumerate}
\item
$B(\xi,2(A_0 \vee (6/c_u + 4)) R_{\xi}) \subsetneq X$
\item
$\frac{8}{c_u} R_{\xi} < \frac{1}{2} \mbox{diam}_{\Omega}(\Omega)$,
if $\mbox{diam}_{\Omega}(\Omega) < \infty$,
%
\item
Any two points in $B_{\widetilde \Omega}\left(\xi,\left(A_0 + \frac{8}{c_u}\right) R_{\xi} \right)$ can be connected by a curve that is $(c_u,C_u)$-inner uniform in $\Omega$. 
\end{enumerate}
\end{definition}

The condition $R_{\xi}>0$ for some boundary point $\xi$ can be understood as a \emph{local} inner uniformity condition.

\begin{remark}
An inner uniform domain $\Omega$ clearly satisfies $R_{\xi}>0$ for all boundary points $\xi \in \widetilde\Omega \setminus \Omega$. The converse is not true.
\end{remark}

\subsection{Green's function estimates}  \label{ssec:GF's estimates}
Recall that $(X,d,\mu,\e,\F)$ is a MMD space that satisfies VD and w-HKE($\Psi$)  and $d$ is geodesic.

\begin{definition} Let $V$ be an open subset of $U$. Set
\begin{align*}
 \F^0_{\mbox{\tiny{loc}}}(U,V) = & \{ f \in L^2_{\mbox{\tiny{loc}}}(V,\mu): \forall \mbox{ open } A \subset V \mbox{ rel.~compact in } \overline{U} \mbox{ with } \\
& \, d_{U}(A,U \setminus V) > 0, \, \exists f^{\sharp} \in \F^0(U): f^{\sharp} = f \, \mu \mbox{-a.e.~on } A \}.
\end{align*}
\end{definition}
Note that $\F^0_{\mbox{\tiny{loc}}}(U,V) \subset \F_{\mbox{\tiny{loc}}}(V)$. If $V=U$, we simply write $\F^0_{\mbox{\tiny{loc}}}(U)$ for $\F^0_{\mbox{\tiny{loc}}}(U,U)$.

\begin{definition}
Let $V \subset U$ be open. We say that a harmonic function $u: V \to \R$ satisfies
\emph{Dirichlet boundary condition} along the boundary of $U$, if
 \[ u \in \F^0_{\mbox{\tiny{loc}}}(U,V). \]
\end{definition}

\begin{remark}
To be precise, a harmonic function $u \in \F^0_{\mbox{\tiny{loc}}}(U,V)$ satisfies Dirichlet boundary condition along  $\widetilde V^{d_U} \setminus U$, where $\widetilde V^{d_U}$ is the completion of $V$ with respect to $d_U$.
\end{remark}

\begin{theorem} \label{thm:basic p^D_B estimate}
Let $B=B(a,R)$ be a ball with $B(a,KR) \subsetneq X$ for some $K>1$.
For any fixed $\epsilon \in (0,1)$ there are constants $c_2, C_2 \in (0,\infty)$ such that for any $x,y \in B(a,R)$,
$t \geq \epsilon \Psi(R)$, the Dirichlet heat kernel $ p^D_B$ is bounded above by
\begin{align*} 
  p^D_B(t,x,y) & \leq \frac{C_2}{V(x,R)} \exp( - c_2 t / \Psi(R) ).
\end{align*}
The constants $c_2,C_2$ depend only on $K$, $C_{\mbox{\em \tiny{VD}}}$, the constants in \emph{w-HKE($\Psi$)}, and $\beta_1$, $\beta_2$, $C_{\Psi}$ in \eqref{eq:Psi beta}. They are independent of $a$ and $R$.
\end{theorem}

\begin{proof}
Follows from changing notation in \cite[Lemma 5.13 part 3]{HSC01}. 
\end{proof}

For a non-empty open set $V \subset X$, the Green function $G_V:V \times V \to [0,\infty]$ is defined by
\[ G_V(x,y) := \int_0^{\infty} p^D_V(t,x,y) dt, \quad x,y \in V. \]
In the present setting, the heat kernel $p^D_V:(0,\infty) \times V \times V \to [0,\infty)$ is continuous by Theorem \ref{thm:w-HKE w-PHI}.

\begin{lemma} \label{lem:4.7}
Let $V \subset X$ be open relatively compact. When $\mbox{\em diam}_d(X)< \infty$, we require in addition that $\mbox{\em diam}_d(V) \leq \frac{1}{2} \mbox{\em diam}_d(X)$. Then for any fixed $x \in V$, the Green function $y \mapsto G_V(x,y)$ is continuous and harmonic on $V \setminus \{x \}$ and belongs to $\F^0_{\mbox{\emph{\tiny{loc}}}}(V,V\setminus\{x\})$.
\end{lemma}

\begin{proof}
We follow the line of reasoning in \cite[Proof of Lemma 4.7]{GyryaSC}. Since we do not have a carr\'e du champ, we use CSA($\Psi$) to control the energy of cutoff functions.

Let $\Omega \subset V$ be open and relatively compact in $\overline{V}$ with $d(\Omega,x) > 0$. Pick finitely many balls $B_i = B(z_i,s_i)$ so that $\Omega \subset \bigcup_i B_i$ and $K = \bigcup_i \overline{B(z_i,4s_i)} \subset X \setminus \{ x \}$. For each $i$, there exists by CSA($\Psi$) a continuous cut-off function  $\psi_i$ with $\psi = 1$ on $B_i$ and $\psi_i = 0$ on $X \setminus B(z_i,2s_i)$. Let $\psi(y) = \min\{1,\sum_i \psi_i(y)\}$ for all $y \in K$, and $f^{\sharp} = \psi G_V(x,\cdot)$. Then $f^{\sharp} = G_V(x,\cdot)$ on $\Omega$. We show that $f^{\sharp}$ is in $\F^0(V)$.

Let $B = B(x,R)$ with $R$ large enough so that $V \subset B$. Since $\mbox{diam}_d(V) \leq \frac{1}{2} \mbox{diam}_d(X)$, we can choose $B$ in such a way that $B(x,\kappa R) \subsetneq X$ for some $\kappa > 1$, so that Theorem \ref{thm:basic p^D_B estimate} is applicable.

Recall that the map $y \mapsto p^D_V(t,x,y)$ is in $\F^0(V)$. The upper bound in w-HKE($\Psi$) and the set monotonicity of the Dirichlet heat kernel, $p^D_V(t,x,y) \leq p^D_B(t,x,y) \leq p(t,x,y)$, provide an upper bound on $p^D_V(t,x,y)$. Integrating over time we obtain that $\psi G_V(x,\cdot) \in L^2(X,\mu)$. By Theorem \ref{thm:basic p^D_B estimate}, there are constants $c, C(B) \in (0,\infty)$ such that for all $t \geq \Psi(R)$ and $y,z \in V$,
\begin{align} \label{eq:4.3}
p^D_V(t,y,z) \leq C(B) e^{-c t / \Psi(R)}.
\end{align}
Moreover, it follows from VD and w-HKE($\Psi$) that there are constants $c',C' \in (0,\infty)$ depending on $K$, $\beta_1$, $\beta_2$, $C_{\Psi}$, and the constants in VD and w-HKE($\Psi$) such that for all $0<t<\Psi(R)$ and $y \in V \cap K$,
\begin{align} \label{eq:4.4}
p^D_V(t,x,y) \leq C' e^{-c' t^{-\frac{1}{\beta_2-1}}}.
\end{align}
This shows that the integral $\psi G_V(x,\cdot) = \int_0^{\infty} \psi p^D_V(t,x,\cdot)dt$ converges at $0$ and $\infty$ in $L^2(X,\mu)$. Hence $\psi G_V(x,\cdot)$ is in $L^2(X,\mu)$.

For fixed $0 < a < b < \infty$, set $g = \int_a^b p^D_V(t,x,\cdot) dt$ and observe that $\psi g, \psi^2 g \in \F^0(V)$. 
Thus, applying CSA($\Psi$), we obtain
\begin{align*}
\int_V d\Gamma(\psi g, \psi g) 
& \leq  \int_{K \cap V} d\Gamma(g,g \psi^2)  +  \int_{K \cap V} g^2 d\Gamma(\psi,\psi) \\
& \leq \int_{K \cap V} d\Gamma(g,g \psi^2) + C \sum_i \int_{B(z_i,2s_i) \cap V} g^2 d\Gamma(\psi_i,\psi_i)  \\
& \leq C' \int_{K \cap V} (-L^D_V g) g \psi^2 \, d\mu + C \sup_i \Psi(2s_i)^{-1} \int_{K \cap V} g^2 d\mu \\
& \leq C' \int_{K \cap V} g \big( p^D_V(a,x,\cdot) - p^D_V(b,x,\cdot) \big) d\mu + C \int_{K \cap V} g^2 d\mu \\
& \leq C' \int_{K \cap V} g \, p^D_V(a,x,\cdot) d\mu 
   + C \int_{K \cap V} g^2 d\mu.
\end{align*}
The constants $C$ and $C'$ change from line to line and depend on $\Omega$, $\Psi$ and the constant $C_S$ from CSA($\Psi$).
Now, observe that \eqref{eq:4.3}-\eqref{eq:4.4} imply that 
\[ \int_{K \cap V} g^2 d\mu = \int_{K \cap V} \left(\int_a^b p^D_V(t,x,\cdot) dt\right)^2 d\mu \]
tends to $0$ when $a$, $b$ tend to infinity or when $a$, $b$ tend to $0$ (this is the same argument we used above to show that $\psi G_V(x,\cdot)$ is in
$L^2(X,\mu)$). The estimates \eqref{eq:4.3}-\eqref{eq:4.4} also imply that $\int_{K \cap V} g p^D_V(a,x,\cdot) d\mu$ tends to $0$ when $a$, $b$ tend to
infinity or when $a$, $b$ tend to $0$. This implies that the integral $\psi G_V(x,y) = \psi \int_0^{\infty} p^D_V(t,x,\cdot)dt$ converges in $\F^0(V)$
as desired.

Next, we show that $G_V(x,\cdot)$ is harmonic on $V \setminus \{ x \}$.
Let $\phi \in D(L^D_V)$ with compact support in $V \setminus \{ x \}$. 
Let $g := \int_0^b p^D_V(t,x,\cdot) dt$.
By w-HKE($\Psi$) and \eqref{eq:4.3}, $p^D_V(b,x,y)$ tends to $0$ as $b \to \infty$ uniformly in $y$.
Hence,
\begin{align*}
\e(G_V(x,\cdot),\phi)
& = - \int \lim_{b \to \infty} g L^D_V \phi d\mu
 = - \lim_{b \to \infty} \int (L^D_V g)  \phi d\mu  \\
& = - \lim_{b \to \infty} \int p^D_V(b,x,\cdot) \phi d\mu
 = 0.
\end{align*}
Since $D(L^D_V)$ is dense in $L^2(V)$, we obtain
$\e(G_V(x,\cdot),\phi) = 0$ for all $\phi \in \F_{\mbox{\tiny{c}}}(V \setminus \{ x \})$.

Finally, we show that $G_V(\cdot,\cdot)$ is jointly continuous on $V \times V \setminus \{(y,y) | y \in V \}$. 
Let $x \in V$. For every $\epsilon >0$ there exists $b > 0$ such that, by \eqref{eq:4.3},
\[ \int_b^{\infty} p(t,x,y) - p(t,x,z) dt \leq \epsilon,
\quad \forall y,z \in V \setminus \{ x \}. \]
Since the Dirichlet heat kernel $p(t,x,y)$ is jointly continuous on $(0,\infty) \times V \times V$, there exists $\delta > 0$ so that for all $z \in V \setminus \{ x \}$ with $d(z,y) < \delta$, we have
\[ \int_0^b p(t,x,y) - p(t,x,z) dt  < \epsilon. \]
This finishes the proof. 
\end{proof}

\begin{lemma} \label{lem:4.8}
Let $B(z,2R) \subsetneq X$. 
\begin{enumerate}
\item
Fix $\theta \in (0,1)$.
There is a constant $C_1$ depending only on $\theta$, $\beta_1$, $\beta_2$, $C_{\Psi}$, and the constants in \emph{VD} and \emph{w-HKE($\Psi$)}, such that for any $x,y \in B(z,R)$ with $d(x,y) \geq \theta R$,
\begin{equation} \label{eq:GF upper estimate in a ball}
  G_{B(z,R)}(x,y) \leq C_1 \frac{ \Psi(R) }{ V(x,R) }.
\end{equation}
\item
Fix $\theta \in (0,1)$. There is a constant $C_2$ depending only on $\theta$, $\beta_1$, $\beta_2$, $C_{\Psi}$ and the constants in \emph{VD} and \emph{w-HKE($\Psi$)}, such that
\begin{equation} \label{eq:GF lower estimate in a ball}
 \forall x,y \in B(z,\theta R), \quad  G_{B(z,R)}(x,y) \geq C_2 \frac{\Psi(R)}{ V(x,R) }.
\end{equation}
\end{enumerate}
\end{lemma}

\begin{proof}
For $\theta \leq \epsilon$, where $\epsilon$ is given by w-LLE($\Psi$), the lower bound \eqref{eq:GF lower estimate in a ball} follows from integrating \eqref{eq:LLE} from $t=\Psi\left(\frac{\theta}{\epsilon} R \right)$ to $t=\Psi(R)$ and applying VD and

For the upper bound \eqref{eq:GF upper estimate in a ball} we use the upper bound in w-HKE($\Psi$) together with the set monotonicity of the Dirichlet heat kernel, and the estimate of Theorem \ref{thm:basic p^D_B estimate},
\begin{align*}
p^D_B(t,x,y) & \leq \begin{cases}    \frac{C_3}{V(x,\Psi^{-1}(t))}  \exp \left( - c_3 \left( \frac{\Psi(d(x,y))}{t} \right)^{1/(\beta_2-1)}\right), 
\quad                                   & t < \frac{1}{2} \Psi(R), \\
                                   \frac{C_2}{V(x,R)} \exp( - c_2 t / \Psi(R) ), \quad  & t \geq \frac{1}{2} \Psi(R),
                    \end{cases}
\end{align*}
where $B = B(z,R)$.
Therefore,
\begin{align*}
\int_{\frac{1}{2}\Psi(R)}^{\infty} p^D_B(t,x,y) dt
&\leq  \int_{\frac{1}{2}\Psi(R)}^{\infty} \frac{C_2}{V(x,R)} \exp( - c_2 t / \Psi(R) ) dt
\leq C \frac{ \Psi(R) }{ V(x,R) },
\end{align*}
and
\begin{align*}
& \int_0^{\frac{1}{2}\Psi(R)} p^D_B(t,x,y) dt \\
&\leq \int_0^{\frac{1}{2}\Psi(R)} \frac{C_3}{V(x,\Psi^{-1}(t))}  \exp \left( - c_3 \left( \frac{\Psi(d(x,y))}{t} \right)^{1/(\beta_2-1)}\right) dt \\
&\leq \frac{C_3}{V(x,R)}\int_0^{\frac{1}{2}\Psi(R)} \frac{V(x,R)}{V(x,\Psi^{-1}(t))}  \exp \left( - c_3 \left( \frac{\Psi(\theta R)}{t} \right)^{1/(\beta_2-1)}\right) dt.
\end{align*}
Due to VD and \eqref{eq:Psi beta}, the integrand is bounded by a constant independent of $R$. Hence, the right hand side is less than or equal to $C \frac{ \Psi(R) }{ V(x,R) }$.
Finally,
\begin{align*}
G_{B(z,R)}(x,y) 
&\leq \int_0^{\frac{1}{2}\Psi(R)} p^D_B(t,x,y) dt + \int_{\frac{1}{2}\Psi(R)}^{\infty} p^D_B(t,x,y) dt 
&\leq C \frac{ \Psi(R) }{ V(x,R) }.
\end{align*}
\end{proof}

\begin{lemma} \label{lem:4.9}
Fix $\theta \in (0,1)$. Let $U \subset X$ be open. Let  $z \in U$, $R>0$ so that $B(z,2R) \subsetneq X$. Set $G_{B_U(z,R)}(x,y)=0$ when $x \notin B_U(z,R)$ or $y \notin B_U(z,R)$.
\begin{enumerate}
\item
There is a constant $C_1$ depending only on $\theta$, $\beta_1$, $\beta_2$, $C_{\Psi}$ and  the constants in \emph{VD} and \emph{w-HKE($\Psi$)}, such that 
\begin{equation} \label{eq:GF upper estimate in U cap B}
 G_{B_U(z,R)}(x,y) \leq G_{U \cap B(z,R)}(x,y) \leq C_1 \frac{ \Psi(R) }{ V(x,R) },
\end{equation}
for all $x,y \in U \cap B(z,R)$ with $d(x,y) \geq \theta R$. 
\item
Let $\delta \in (0,1/3)$. Suppose any two points in $B_U(z,\delta R)$ can be 
connected by a $(c_u,C_u)$-inner uniform curve in $U$. Then there is a constant $C_2$ depending only on $c_u$, $C_u$, $\theta$, $\beta_1$, $\beta_2$, $C_{\Psi}$ and the constants in \emph{VD} and \emph{w-HKE($\Psi$)}, such that 
\begin{equation} \label{eq:GF lower estimate in U cap B}
 G_{B_U(z,R)}(x,y) \geq C_2 \frac{ \Psi(R) }{ V(x,R) },
\end{equation}
for all $x,y \in B_U(z, \delta R)$ with $d(x, X \setminus U)$, $d(y, X \setminus U) \in (\theta R, \infty)$ and $d_U(x,y) \leq \delta R / C_u$.
\end{enumerate}
\end{lemma}

\begin{proof}
We follow the line of reasoning of \cite[Lemma 4.9]{GyryaSC}. 
(i) Set $B = B(z,R)$, $W = U \cap B(z,R)$. The upper bound \eqref{eq:GF upper estimate in U cap B} follows easily from Lemma \ref{lem:4.8} and the monotonicity inequality $G_W \leq G_B$. 
(ii) By assumption, there is an $\epsilon_1 > 0$ such that, for any $x,y \in B_U(z,\delta R)$ satisfying
$d(x,X \setminus U)$, $d(y,X \setminus U) > \theta R$, there is a path in $U$ from $x$ to $y$ of length less than
$C_u d_U(x,y) \leq \delta R$ that stays at distance at least $\epsilon_1 R$ from $X \setminus U$. Since $x,y \in B_U(z,\delta R)$ and $\delta < 1/3$, this path is contained in $B_U(z,R) \cap \{ \zeta \in U : d(\zeta,X \setminus U) > \epsilon_1 R \}$.
Using this path, the elliptic Harnack inequality EHI easily reduces the lower bound \eqref{eq:GF lower estimate in U cap B} to
the case when $y$ satisfies $d(x,y) = \eta R$ for some arbitrary fixed $\eta \in (0,\epsilon_1)$ small enough.
Pick $\eta > 0$ so that, under the hypotheses of the lemma, the ball $B_x = B(x,2\eta R)$ is
contained in $B_U(z,R)$. Let $W=B_U(z,R)$. Then the monotonicity property of Green functions implies that $G_W(x,y) \geq G_{B_x}(x,y)$. Lemma \ref{lem:4.8} and the volume doubling property then yield
 $G_W(x,y) \geq C_2 \frac{\Psi(R)}{V(x,R)}$.
This is the desired lower bound.
\end{proof}

\subsection{The harmonic measure}
We continue to work under the assumption that VD and w-HKE($\Psi$) are satisfied.

By \cite[Theorem 7.2.2]{FOT94}, the locality of $(\e,\F)$ implies that there exists a diffusion $((\mathcal{X}_t)_{t \geq 0}, (\mathbb{P}_x)_{x \in X}, (\F_t)_{t \geq 0})$ associated with $(\e,\F)$.

The following proposition does not require VD or w-HKE($\Psi$) to hold, as long as the heat kernel exists and is jointly continuous on $(0,\infty) \times X \times X$.
\begin{proposition}
The Markovian strongly continuous semigroup $(P_t)$, $t>0$, corresponding to $(\e,\F)$ has the Feller property.
In particular, the transition densities of the diffusion process $((\mathcal{X}_t)_{t \geq 0}, (\mathbb{P}_x)_{x \in X}, (\F_t)_{t \geq 0})$ are given by the heat kernel.
\end{proposition}

\begin{proof}
Let $f \in C_{\infty}$, the space of continuous functions that vanish at infinity. Since $P_t f(x) = \int p(t,x,y) f(y) \mu(dy)$ for all $x \in X$, $t>0$, and $p(t,x,y)$ is jointly continuous on $(0,\infty) \times X \times X$, it follows easily that, for any $t >0$, $x \mapsto P_t f(x)$ is continuous.

It remains to show that $\lim_{t \to 0} P_tf(x) = f(x)$ for all $x \in X$.
Observe that $f$ is bounded and uniformly continuous.
Let $\epsilon>0$ and $\delta>0$ so that $|f(y) -f(x)|<\epsilon$ for all $x,y \in X$ with $d(x,y) < \delta$.

Let $t>0$, $x \in X$. Since $P_t 1 = 1$,
\begin{align*}
& |P_t f(x) - f(x)| \\
& =
\left|\int_X (f(y) - f(x)) p(t,x,y) \mu(dy)\right| \\
& \leq 
\int_{B(x,\delta)} |f(y) - f(x)| p(t,x,y) \mu(dy) + \int_{X \setminus B(x,\delta)} |f(y) - f(x)| p(t,x,y) \mu(dy) \\
& \leq
\epsilon  
+ \frac{2 c_1 ||f||_{\infty}}{V(x,\Psi^{-1}(t))} \int_{X \setminus B(x,\delta)} \exp \left(-c_2 \left( \frac{\Psi(d(x,y))}{t}\right)^{\frac{1}{\beta_2-1}} \right) \mu(dy) \\
& \leq
\epsilon  
+ {2 c_1 ||f||_{\infty}}{V(x,\Psi^{-1}(t))} \exp \left(- \frac{c_2}{2} \left( \frac{\Psi(\delta)}{t}\right)^{\frac{1}{\beta_2-1}} \right) \\
& \quad \int_{X \setminus B(x,\delta)}  \exp \left(- \frac{c_2}{2} \left( \frac{\Psi(d(x,y))}{t}\right)^{\frac{1}{\beta_2-1}} \right) \mu(dy).
\end{align*}
Since ${2 c_1 ||f||_{\infty}}{V(x,\Psi^{-1}(t))} \exp \left(- \frac{c_2}{2} \left( \frac{\Psi(\delta)}{t}\right)^{\frac{1}{\beta_2-1}} \right)$ tends to $0$ as $t \to 0$, it suffices to that that the integral on the right hand side is bounded.
For $t<1$, we have
\begin{align*}
& \int_{X \setminus B(x,\delta)}  \exp \left(- \frac{c_2}{2} \left( \frac{\Psi(d(x,y))}{t}\right)^{\frac{1}{\beta_2-1}} \right) \mu(dy) \\
& \leq
\sum_{n=1}^{\infty} \int_{B(x,2^n \delta) \setminus B(x,2^{n-1} \delta)}   \exp \left(- \frac{c_2}{2} \left( \Psi(2^{n-1}\delta)\right)^{\frac{1}{\beta_2-1}} \right) \mu(dy) \\
& \leq
V(x,\delta) \sum_{n=1}^{\infty}  2^n \exp \left( - c 2^n \right) < \infty,
\end{align*}
where $c$ depends on $\delta$, $\beta_1$, $\beta_2$, $C_{\Psi}$, $C_{\mbox{\tiny{VD}}}$.
Hence, $\lim_{t\to 0}P_tf(x) = f(x)$ for all $x \in X$. Since
\[ {V(x,\delta)}{V(x,\Psi^{-1}(t))} \leq C \left(\frac{\delta}{\Psi^{-1}(t)}\right)^{\nu}, \quad \forall x \in X \]
for some $\nu,C$ depending on $C_{\mbox{\tiny{VD}}}$, we see from the above argument that the convergence $P_tf \to f$ is uniform in $x$. This shows that $P_t f$ vanishes at infinity.
\end{proof}
%
Set

Let $B \subset X$ be open relatively compact in $V \subset X$. Let $U := X \setminus B$.

\begin{definition}
Let $\sigma_U := \inf \{ t>0 : \mathcal{X}_t \in U \}$ be the first hitting time of the set $U$.
The \emph{hitting distribution} $H_U$ is defined as
\begin{align*}
H_U(x,A) := \mathbb{P}_x[\mathcal{X}_{\sigma_U} \in A, \sigma_U < \infty], \qquad x \in X.
\end{align*}
\end{definition}

For $u \in \F_e$, let
\begin{align} \label{eq:H_Uf}
H_U \widetilde u(x) := \mathbb{E}_x[\widetilde u(\mathcal{X}_{\sigma_U}), \sigma_U < \infty], \qquad x \in X.
\end{align}

\begin{proposition} \label{prop:H_U q-cts}
For any $u \in \F_e$, $H_U \widetilde u$ is in $\F_e$, $H_U \widetilde u$ is quasi-continuous, and 
\[ \e(H_U \widetilde u,v) = 0, \qquad \forall v \in \F(B)_e. \]
\end{proposition}
\begin{proof}
See \cite[Theorem 4.6.5]{FOT94}.
\end{proof}

In view of Proposition \ref{prop:H_U q-cts}, it is not hard to derive a 0-order version of \cite[Lemma 4.5.1]{FOT94}. Hence, the locality of $(\e,\F_e)$ implies that the hitting distribution $H_U(x,A)$ is supported on the boundary $\partial_X B$ for q.e.~$x \in B$.
The measure
\[ \omega(x,A,B) := H_{X \setminus B}(x,A), \quad A \subset \partial_X B, \]
on $\partial_X B$ is called the \emph{harmonic measure}. 

\begin{lemma} \label{lem:harm meas cts}
Suppose that \mbox{\em w-HKE($\Psi$)} holds.
For any Borel measurable $A \subset \partial_X B$, the map
\[ x \mapsto \omega(x,A,B) \]
is continuous and harmonic on $B$. Moreover, for any $x \in X$, $\omega(x,\cdot,B)=0$ charges no set of zero capacity. 
\end{lemma}  
  
\begin{proof}
Denote the map $x \mapsto \omega(x,A,B)$ by $u$. 
First, show that $u|_B$ is in $\F_{\mbox{\tiny{loc}}}(B)$ and $u$ is harmonic by applying \cite[Theorem 6.7.11]{CF12}. It suffices to verify that $\{ u(\mathcal{X}_{t \wedge \sigma_U}), t \geq 0 \}$ is a uniformly integrable $\mathbb{P}_x$-martingale for q.e.~$x \in X$.

We have $\mathbb{P}_x[\sigma_U = 0]=1$ for all regular boundary points $x \in U$. The set of irregular points is exceptional by \cite[Theorem A.2.6(i) and Theorem 4.1.3]{FOT94}. Therefore,
\[ u(\mathcal{X}_{t \wedge \sigma_U}) = u(\mathcal{X}_0) = u(x) \qquad \mathbb{P}_x \mbox{-a.s.~for q.e.~} x \in U, \]
which is a $\mathbb{P}_x$-martingale.
Next, let $x \in B$ and $0 \leq s < t < \infty$.
Then
\begin{align*}
 \mathbb{E}_x[u(\mathcal{X}_{t \wedge \sigma_U}) | \F_s] 
& =  \mathbb{E}_x[ \mathbb{E}_{\mathcal{X}_{t \wedge \sigma_U}}[\mathbf{1}_A(\mathcal{X}_{\sigma_U})] | \F_s]
=  \mathbb{E}_x[ \mathbb{E}_{\mathcal{X}_{s \wedge \sigma_U}}[\mathbf{1}_A(\mathcal{X}_{\sigma_U})] ] \\
& =  \mathbb{E}_x[u(X_{s \wedge \sigma_U})], \qquad \mathbb{P}_x  \mbox{-.a.s.}.
\end{align*}
Hence $u(\mathcal{X}_{t \wedge \sigma_U})$, $t \geq 0$, is a $\mathbb{P}_x$-martingale.
The uniform integrability is immediate, because $u$ is bounded. We have proved that $u$ is harmonic.

By EHI, $u$ admits a modification $\widetilde u$ which is continuous on $B$. Then, for any $x \in B$,
\begin{align*}
P^B_t \widetilde u(x)
& = \int_B p(t,x,y) \widetilde u(y) d\mu
 = \int_B p(t,x,y) u(y) d\mu
 = \mathbb{E}_x [ u(\mathcal{X}_t) \mathbf{1}_{\{t < \sigma_U \}} ] \\
& = \mathbb{E}_x [ \mathbb{E}_{\mathcal{X}_t} [\mathbf{1}_A(\mathcal{X}_{\sigma_U})] \mathbf{1}_{\{t < \sigma_U \}} ].
\end{align*}
Hence, by the strong Markov property,
\begin{align*}
P^B_t \widetilde u(x)
& = \mathbb{E}_x [ \mathbb{E}_x [\mathbf{1}_A(\mathcal{X}_{\sigma_U} \circ \theta_t) | \F_t ] \mathbf{1}_{\{t < \sigma_U \}} ] 
= \mathbb{E}_x [\mathbf{1}_A(\mathcal{X}_{\sigma_U} \circ \theta_t) \mathbf{1}_{\{t < \sigma_U \}} ] \\
& = \mathbb{E}_x [\mathbf{1}_A(\mathcal{X}_{\sigma_U}) \mathbf{1}_{\{t < \sigma_U \}} ].
\end{align*}
Letting $t \downarrow 0$, and using the continuity of $H$, we obtain
\[ \widetilde u(x) = \lim_{t \downarrow 0}P^B_t \widetilde u(x) = \lim_{t \downarrow 0} \mathbb{E}_x [\mathbf{1}_A(\mathcal{X}_{\sigma_U}) \mathbf{1}_{\{t < \sigma_U \}} ] = u(x). \]
This proves that $u$ is continuous. 

The fact that for any $x \in X$, $\omega(x,\cdot,B)=0$ charges no set of zero capacity, follows from \cite[Theorem 4.3.3]{FOT94}.
\end{proof}  

\begin{proposition}[Mean value property]
If $u$ is harmonic and continuous on $V$, then 
\begin{align} \label{eq:mean value property}
 u(x)
 = \int_{\partial_X B} u(y) \, \omega(x,dy,B), \qquad \mbox{ for all } x \in B.
\end{align}
\end{proposition}

\begin{proof}
If $u$ is harmonic on $V$, then $u(x) = H_{U} u(x)$ for q.e.~$x \in B$ by \cite[Theorem 6.7.9]{CF12}. Hence the proposition follows from Lemma \ref{lem:harm meas cts}.
\end{proof} 


The following monotonicity property of the harmonic measure follows from the fact that $(\mathcal{X}_t)_{t \geq 0}$ has continuous paths. If $A \subset B$ are open subsets of $X$ and $D \subset X$ is open, then 
\begin{align} \label{eq:monotonicity of harm meas}
\omega(x, D \cap \partial_X B, D \cap B) \leq \omega(x, D \cap \partial_X A, D \cap A), \qquad \forall x \in D \cap A.
\end{align}

\subsection{Maximum principle} \label{ssec:max principle}
Let $\Omega \subset X$ be open. In this subsection, we do not assume inner uniformity.

\begin{proposition}[Maximum principle] \label{prop:max principle}
\begin{enumerate}
Suppose $(\e^D_{\Omega},\F^0(\Omega))$ is transient.
Let $B = B_{\Omega}(x,r)$ be an inner ball such that $B_{\Omega}(x,2r) \subsetneq \Omega$. 
\item
Let $f$ be a non-negative harmonic function on $B_{\Omega}(x,2r)$ with Dirichlet boundary condition along $\partial_{\widetilde\Omega} \Omega \cap B_{\widetilde\Omega}(x,2r)$. 
Then $f$ has a modification $\widetilde f$ which is continuous on $B_{\Omega}(x,3r/2)$ and satisfies
\[ \sup_{B} \widetilde f = \sup_{\partial_{\Omega}B} \widetilde f. \]
\item
Let $f$ be a non-negative harmonic function on $\Omega \setminus B_{\Omega}(x,r)$ with Dirichlet boundary condition along $\partial_{\widetilde\Omega} \Omega \setminus B_{\widetilde\Omega}(x,r)$. 
Then $f$ has a modification $\widetilde f$ which is continuous on $\Omega \setminus B_{\Omega}(x,3r/2)$ and satisfies
\[ \sup_{\Omega \setminus B_{\Omega}(x,2r)} \widetilde f = \sup_{\partial_{\Omega}B_{\Omega}(x,2r)} \widetilde f. \]
\end{enumerate}
\end{proposition}

\begin{proof} 
Since the proof of (ii) is very similar to that of (i), we only show (i). 
Let $f$ be as required in (i). Then there is a function $f^{\sharp} \in \F^0(\Omega)$ such that $f=f^{\sharp}$ a.e.~on $B_{\Omega}(x,3r/4)$. Replacing $f$ by $f^{\sharp}$ if necessary, it suffices to consider the case when $f \in \F^0(\Omega)$. Recall from Theorem \ref{thm:VD+PI+CSA} that EHI holds. By a standard argument, (e.g., \cite[Lemma 5.2]{GT12})
 the elliptic Harnack inequality implies the H\"older continuity of harmonic functions. Hence, $f$ admits a modification $\widetilde f$ which is continuous on $B_{\Omega}(x,3r/2)$, in particular on the part of the boundary $\partial_{\Omega} B$.  Let $(\mathcal{X}_t)$, $t>0$, be the diffusion process associated with $(\e^D_{\Omega},\F^0(\Omega))$ and let $\F^0(\Omega)_e$ be the extended Dirichlet space. 
 
Set $U = \Omega \setminus B$.   
It is immediate from its definition \eqref{eq:H_Uf} that $H_U \widetilde f$ satisfies the maximum principle. We show that $H_U \widetilde f$ is a quasi-continuous modification of $f$. Since $(\e^D_{\Omega},\F^0(\Omega))$ is transient, its extended Dirichlet space admits the orthogonal decomposition
\begin{align*}
& \F^0(\Omega)_e = \F^0(\Omega)_{e, \Omega \setminus U} \oplus \mathcal{H}_U, \\
& \F^0(\Omega)_{e, \Omega \setminus U} = \big\{ u \in \F^0(\Omega)_e: \widetilde u = 0 \mbox{ q.e. on } U \big\}. 
\end{align*}
Let $\mathcal{P}_{\mathcal{H}_U}:\F^0(\Omega)_{e} \to \mathcal{H}_U$ be the orthogonal projection.
By \cite[Theorem 4.3.2]{FOT94}, $H_U \widetilde f$ is a quasi-continuous modification of $\mathcal{P}_{\mathcal{H}_U}f$.
 
Since $f$, hence also $-f$, is harmonic on $U$, and by the $0$-order analogue of \cite[Lemma 2.2.6 (iii) $\Rightarrow$ (ii)]{FOT94}, we have
\[ \e(f,v) = 0, \quad \forall v \in \F^0(\Omega)_{e, \Omega \setminus U}. \]
That is, $f \in \mathcal{H}_U$. Hence, $f = \mathcal{P}_{\mathcal{H}_U}f$. We have proved that $f$ admits the quasi-continuous modification $H_U \widetilde f$ which satisfies the maximum principle.
\end{proof}

\subsection{Representation formula for harmonic functions}

An essential step in the proof of the boundary Harnack principle is the reduction to Green functions estimates, see Section \ref{ssec:reduction to GF}, through a representation formula for harmonic functions. On Euclidean space, a representation formula was used for this purpose by Aikawa (\cite{Aik01}, \cite[p.331]{ALM03}). In \cite{GyryaSC} it is claimed that the representation extends to (non-fractal) symmetric Dirichlet spaces. Note that classic potential theoretic references require the domain of integration to be compact (see, e.g.,~\cite[Lemma 6.2]{GH13}), which is not the case in \eqref{eq:representation formula} below. Thus, we include a full proof. 

\begin{proposition} \label{prop:representation formula}
Let $\Omega$ be a connected open subset of $X$ and suppose that $(\e^D_{\Omega},\F^0(\Omega))$ is transient. Let $r,\epsilon \in (0,\infty)$ and $\xi \in \widetilde \Omega \setminus \Omega$. Let $f \in \F^0_{\mbox{\em \tiny{loc}}}(\Omega,B_{\Omega}(\xi,(1+2\epsilon)r))$ and suppose that $f$ is non-negative harmonic on $B_{\Omega}(\xi,(1+\epsilon)r)$. Then $f$ admits a $\mu$-version $\widetilde f$ which is continuous on $B_{\Omega}(\xi,(1+\epsilon)r)$ and satisfies
\begin{align} \label{eq:representation formula}
\widetilde f(x) = \int_{\partial_{\Omega} B_{\Omega}(\xi,r)} G_{B_{\Omega}(\xi,(1+\epsilon)r)} (x,y) \nu_f(dy), \qquad \mbox{ for any } x \in B_{\Omega}(\xi,r),
\end{align}
for some Radon measure $\nu_f$ on $B_{\Omega}(\xi,(1+\epsilon)r)$ with $\nu_f \big( B_{\Omega}(\xi,(1+\epsilon)r) \setminus \partial_{\Omega} B_{\Omega}(\xi,r) \big)=0$.
\end{proposition}

\begin{proof}
Let $Y := B_{\Omega}(\xi,(1+\epsilon)r)$.
Since $f \in \F^0_{\mbox{\tiny{loc}}}(\Omega,B_{\Omega}(\xi,(1+2\epsilon)r))$, there exists $f^{\sharp} \in \F^0(\Omega)$ so that $f=f^{\sharp}$ $\mu$-a.e.~on $Y$. Since $f \geq 0$, we may assume that $f^{\sharp} \geq 0$ $\mu$-a.e.~on $\Omega$. 
By \cite[Lemma 6.1(a)]{GH13}, $P^{D,Y}_t f^{\sharp} \leq f^{\sharp}$ $\mu$-a.e. on $Y$ for all $t > 0$. Then \cite[Corollary 2.4]{Ou96} applied to the closed convex subset
$K := \{ u \in L^2(Y,\mu) | u \leq f^{\sharp} \mbox{ $\mu$-a.e.~on } Y \}$ of $L^2(Y,\mu)$ implies that 
\[ \forall u \in \F^0(Y), \quad u \wedge f^{\sharp} \in \F^0(Y) \quad  \mbox{ and } \quad \e(u \wedge f^{\sharp}, u \wedge f^{\sharp}) \leq \e(u,u). \]
This can easily be extended to
\begin{align} \label{eq:u wedge f^{sharp}}
 \forall u \in \F^0(Y)_e, \quad u \wedge f^{\sharp} \in \F^0(Y)_e \quad  \mbox{ and } \quad \e(u \wedge f^{\sharp}, u \wedge f^{\sharp}) \leq \e(u,u).
 \end{align}

Set $A = \{ x \in \Omega : d_{\Omega}(\xi,x) \leq r \}$ 
Let $g \in \F^0(Y)_e$ be such that 
\begin{align} \label{eq:g min}
 \e(g,g) = \inf \left\{ \e(u,u) \Big| u \in \F^0(Y)_e, \widetilde u \geq \widetilde{f^{\sharp}} \mbox{ q.e.~on } A \right\}.
\end{align}
Existence and uniqueness of $g$ can be shown in the same way as \cite[Proof of Lemma 2.1.1]{FOT94} by using the $0$-order version of \cite[Theorem 2.1.4(i)]{FOT94}.
Note that $g \geq f^{\sharp}$ on $A$.
Applying \eqref{eq:u wedge f^{sharp}} to $u=g$, we obtain that $g \wedge f^{\sharp}$ also obtains the minimum in \eqref{eq:g min}, hence $g = g \geq f^{\sharp}$. Therefore, $g = f^{\sharp}$ q.e.~on $A$.

Hence, for any $v \in \F^0(Y)_e$ with $\widetilde v \geq 0$ q.e.~on $A$, we have $\e(g+tv,g+tv) \geq \e(g,g)$ for all $t>0$. Thus, $\e(g,v) \geq 0$.

Set $F = \partial_{\Omega}B_{\Omega}(\xi,r)$. For every $u \in \F \cap C_0(X)$ with support in $Y \setminus F$, it follows from a cut-off function argument that $u \mathbf{1}_{B_{\Omega}(\xi,r)}$ and $u \mathbf{1}_{Y \setminus A}$ are in $\F \cap C_0(X)$, because $\inf \{ d(x,y) | x \in \mbox{supp}(u \mathbf{1}_{B_{\Omega}(\xi,r)}), y \in \mbox{supp}(u \mathbf{1}_{Y \setminus A}) \} > 0$.
Since $g = f^{\sharp}$ is harmonic on $Y$, we have $\e(g,u \mathbf{1}_{B_{\Omega}(\xi,r)}) = 0$. Applying the result of the previous paragraph to $v=u \mathbf{1}_{Y \setminus A}$, we find that $\e(g,u \mathbf{1}_{Y \setminus A}) = 0$. Now for $v \in \F^0(Y)_e$ with $\widetilde v \geq 0$ q.e.~on $F$, we have $\widetilde v - \widetilde{v^+} = 0$ q.e.~on $F$ and hence $\e(g,v-v^+)=0$ by \cite[Theorem 2.3.3(ii)]{FOT94}. By the result of the previous paragraph, $\e(g,v) = \e(g,v^+) \geq 0$. 

Now it follows from the $0$-order version of \cite[Lemma 2.2.6]{FOT94} 
that $g$ is a ($0$ order) potential for some positive Radon measure $\nu$ supported on $F$. By \cite[Theorem 2.2.5]{FOT94},
\begin{align} \label{eq:g is potential} \e^D_Y(g,v) = \int_{\partial_{\Omega}B(\xi,r)} \widetilde{v}(x) \nu(dx), \quad \forall v \in \F^0(Y)_e.
\end{align}
Let $G_Y$ be the Green operator defined by \cite[(1.5.3)]{FOT94} associated with $(\e^D_Y,\F^0(Y))$.
We will apply \eqref{eq:g is potential} to $v = G_{Y}\phi$ for any test function $\phi:=u h$ with $u$ being non-negative bounded and $h$ the reference function (see \cite[p.35]{FOT94}) of the transient semigroup associated with $(\e^D_Y,\F^0(Y)_e)$. Then $\int_Y \phi G_Y \phi d\mu < \infty$. By \cite[Theorem 1.5.4]{FOT94} we obtain that $v \in \F^0(Y)$, and 
\begin{align*}
\int_Y  g(x) \phi(x) \mu(dx)  
  = \e^D_Y(g,v) 
 & = \int_F \int_Y G_{Y}(y,x) \phi(x) \mu(dx) \nu(dy) \\
 & = \int_Y \left( \int_F G_{Y}(y,x) \nu(dy) \right) \phi(x) \mu(dx).
\end{align*}
Here we used that $\int_Y G_Y(\cdot,y) \phi(y) \mu(dy)$ is a quasi-continuous $\mu$-version of $G_Y\phi$. By definition of the Green operator and the fact that the semigroup admits the heat kernel, it is clear that $G_Y$ admits as kernel the Green function $G_Y(\cdot,\cdot)$. The quasi-continuity follows from Lemma \ref{lem:4.7} and the dominated convergence theorem.

Hence, 
\[  g(x) = \int_F G_Y(y,x) \nu(dy) \quad \mbox{ for a.e. } x \in Y. \]
Since $f= g$ $\mu$-a.e.~on $B_{\Omega}(\xi,r)$, the assertion follows for almost every $x \in B_{\Omega}(\xi,r)$. Since $f$ is harmonic, it satisfies EHI, hence admits a continuous modification $\widetilde f$ (see, e.g., \cite[Lemma 5.2]{GT12}). Also, for any $y \in \F$, the Green function $G(\cdot,y)$ is continuous on $B_{\Omega}(\xi,r)$ by Lemma \ref{lem:4.7}, hence $g$ is continuous by the dominated convergence theorem. Thus, the assertion follows.
\end{proof}

\section{Boundary Harnack Principle}
\label{sec:BHP}

\subsection{Reduction to Green functions estimates} \label{ssec:reduction to GF}
Let $(X,d,\mu,\e,\F)$ be a MMD space that satisfies VD and w-HKE($\Psi$)  and $d$ is geodesic.

Let $\Omega$ be an open connected subset of $X$.
Fix $c_u \in (0,1)$ and $C_u \in (1,\infty)$. Let $A_3 = 2(12 + 12C_u)$, $A_0 = A_3 + 7$. 
We do not require the inner uniformity of the whole domain $\Omega$, but assume that $\Omega$ is locally $(c_u,C_u)$-inner uniform near a fixed boundary point $\xi$ in the sense of Definition \ref{def:R_xi}. That is, we will assume that $R_{\xi}>0$.  We use the notation of Section \ref{ssec:uniform domains}.

Recall that for an open set $Y \subset \Omega$, $ G_{Y}$ is the Green function corresponding to the transient Dirichlet form $( \e^D_{Y},\F^0(Y))$.

\begin{theorem} \label{thm:bHP with GF's}
There exists a constant $A_1' \in (1,\infty)$ such that for any $\xi \in \partial_{\widetilde{\Omega}} \Omega$ with $R_{\xi}>0$ and any
 \[ 0 < r < \left(1-\frac{7}{A_0 \vee \frac{12}{c_u}} \right) R_{\xi}. \]
we have
 \[ \frac{  G_{Y} (x,y) }{   G_{Y} (x',y) }
     \leq A_1' \frac{  G_{Y} (x,y') }{  G_{Y} (x',y') }, \]
for all $x,x' \in B_{\Omega}(\xi,r)$ and $y,y' \in \partial_{\Omega} B_{\Omega}(\xi,6r)$. 
Here, $Y = B_{\Omega}(\xi,A_0r)$, and .
The constant $A'_1$ depends only on $c_u$, $C_u$, $\beta_1$, $\beta_2$, $C_{\Psi}$ and the constants in \emph{VD} and \emph{w-HKE($\Psi$)}.
\end{theorem}
The proof of this theorem is the content of Section \ref{sec: BHP for L} below. It is based on the estimates for the Green's functions in Section \ref{ssec:GF's estimates} and the maximum principles of Proposition \ref{prop:max principle}.

\begin{remark}
$r < \left(1-\frac{7}{\left(A_0 + \frac{8}{c_u} \right)} \right) R_{\xi}$ implies that $r < R_{\xi'}$ for all $\xi' \in \partial_{\widetilde{\Omega}} \Omega \cap B_{\widetilde \Omega}(\xi,7R_{\xi})$.
\end{remark}

\begin{theorem} \label{thm:bHP for u}
There exists a constant $A_1 \in (1,\infty)$ such that for any $\xi \in \partial_{\widetilde{\Omega}} \Omega$ with $R_{\xi}>0$, any
 \[ 0 < r < \left( 1-\frac{7}{\left(A_0 + \frac{8}{c_u}\right)} \right) R_{\xi}, \]
and any two functions $u,v$ that are non-negative harmonic on $B_{\Omega}(\xi,A_0 r)$ with Dirichlet boundary condition along $(\partial_{\widetilde{\Omega}} \Omega) \cap B_{\widetilde \Omega}(\xi,2A_0r)$, we have
 \[  \frac{u(x)}{u(x')} \leq A_1 \frac{v(x)}{v(x')}, \]
for all $x, x' \in B_{\Omega}(\xi,r)$. The constant $A_1$ depends only on $c_u$, $C_u$, $\beta_1$, $\beta_2$, $C_{\Psi}$ and the constants in \emph{VD} and \emph{w-HKE($\Psi$)}.
\end{theorem}

\begin{proof}
Fix $\xi \in \partial_{\widetilde{\Omega}} \Omega$ and $r > 0$ as in the theorem. Let $Y := B_{\Omega}(\xi,A_0r)$. By the representation formula of Proposition \ref{prop:representation formula}, there exists a Borel measure $\nu_u$ such that
\begin{align} \label{eq:u is a potential}
 u(z) = \int_{\partial_{\Omega} B_{\Omega}(\xi,6r)}  G_{Y}(z,y) \nu_u(dy)
\end{align}
for all $z \in B_{\Omega}(\xi,6r)$.

By Theorem \ref{thm:bHP with GF's}, there exists a constant $A_1' \in (1,\infty)$ such that for all $x,x' \in B_{\Omega}(\xi,r)$ and all $y,y' \in \partial_{\Omega} B_{\Omega}(\xi,6r)$, we have
 \[ \frac{  G_{Y}(x,y) }{  G_{Y}(x',y) } \leq A'_1 \frac{  G_{Y}(x,y') }{  G_{Y}(x',y') }. \]
For any (fixed) $y' \in \partial_{\Omega} B_{\Omega}(\xi,6r)$, we find that
\begin{align*}
\frac{1}{A'_1} u(x)  
& \leq  \frac{  G_{Y}(x,y') }{  G_{Y}(x',y') }   \int_{\partial_{\Omega} B_{\Omega}(\xi,6r)}   G_{Y}(x',y) \nu_u(dy) \\
& =   \frac{  G_{Y}(x,y') }{  G_{Y}(x',y') }  u(x')
   \leq  A'_1 u(x).
\end{align*}
We get a similar inequality for $v$. Thus, for all $x, x' \in B_{\Omega}(\xi,r)$,
\begin{equation} \label{eq:bHP with u and G}
\frac{1}{A'_1} \frac{ u(x) }{ u(x') } 
    \leq \frac{  G_{Y}(x,y') }{ G_{Y}(x',y') }   \leq   A'_1 \frac{ v(x) }{ v(x') }.
\end{equation}
\end{proof}

\begin{corollary}[Carleson estimate]
Let $\xi \in \partial_{\widetilde{\Omega}} \Omega$ with $R_{\xi}>0$ and $r \in (0,R_{\xi}/3)$. Let $u$ be a non-negative harmonic function on $B_{\Omega}(\xi,A_0r)$ with Dirichlet boundary condition along $(\partial_{\widetilde{\Omega}} \Omega) \cap B_{\widetilde \Omega}(\xi,6r)$. Then, for any $x \in B_{\Omega}(\xi,r/2)$,
 \[  u(y) \leq C u(x_r) \quad \forall y \in B_{\Omega}(x,r/2), \]
 where $x_r$ is as in Lemma \ref{lem:x_r}.
The constant $C$ depends only on $c_u$, $C_u$, $\beta_1$, $\beta_2$, $C_{\Psi}$ and the constants in \emph{VD} and \emph{w-HKE($\Psi$)}.
\end{corollary}

\begin{proof}
This can be proved by applying the boundary Harnack principle of Theorem \ref{thm:bHP for u} to $u$ and a Green function on a suitably chosen ball, and then using the Green function estimates of Lemma \ref{lem:4.9}. The proof is omitted since it is similar to the proof of \cite[Theorem 4.17]{GyryaSC}.
\end{proof}

\subsection{Proof of Theorem \ref{thm:bHP with GF's}} \label{sec: BHP for L}

We follow closely \cite{Aik01}, \cite{ALM03}, \cite{GyryaSC} and \cite{LierlSC1}.
Let $\Omega$ be as above and fix $\xi \in \partial_{\widetilde{\Omega}} \Omega$ with $R_{\xi} > 0$.

\begin{definition}
For $\eta \in (0,1)$ and any open $U \subset X$, define the \emph{capacitary width} $w_{\eta}(U)$ by
 \[ w_{\eta}(U) = \inf \left\{ r>0 : \forall x \in U, \frac{ {\mbox{Cap}}_{B(x,2r)} \big( \overline{B(x,r)} \setminus U \big) }{ {\mbox{Cap}}_{B(x,2r)} \big( \overline{B(x,r)} \big) } \geq \eta \right\}, \]
with the convention that $w_{\eta}(U) = +\infty$ when the infimum is taken over the empty set (e.g.~when $\mbox{Cap}_{B(x,2r)}$ is not well-defined).
\end{definition}
Note that $w_{\eta}(U)$ is an increasing function of $\eta \in (0,1)$ and an increasing function of the set $U$.

\begin{lemma} \label{lem:4.12} 
There are constants $A_7,\eta \in (0,\infty)$ depending only on $c_u$, $C_u$, $\beta_1$, $\beta_2$, $C_{\Psi}$ and the constants in \emph{VD} and \emph{w-HKE($\Psi$)}, such that for all $0 < r < 2 R_{\xi}$, 
 \[ w_{\eta}(\{ y \in B_{\Omega}(\xi,A_0 r): d_{\Omega}(y, \partial_{\widetilde{\Omega}} \Omega) < r \}) \leq A_7r. \]
\end{lemma}

\begin{proof} We follow \cite[Lemma 4.12]{GyryaSC}. 
Let $Y_r = \{ y \in B_{\Omega}(\xi,A_0 r): d_{\Omega}(y, \partial_{\widetilde{\Omega}} \Omega) < r \}$ and $y \in Y_r$. Since $\frac{4}{c_u} r < \frac{1}{2} \mbox{diam}_{\Omega}(\Omega)$, there exists a point $x \in \Omega$ such that
$d_{\Omega}(x,y) = 4r / c_u$. There exists an inner uniform curve connecting $y$ to $x$ in $\Omega$. Let $z \in \partial_{\Omega} B_{\Omega}(y,2r/c_u)$ be a point 
on this curve and note that $d_{\Omega}(y,z) = 2r/c_u \leq d_{\Omega}(x,y) - d_{\Omega}(y,z) \leq d_{\Omega}(x,z)$. Hence,
 \[ d_{\Omega}(z, \partial_{\widetilde{\Omega}} \Omega)  \geq  c_u \min\{d_{\Omega}(y,z),d_{\Omega}(z,x)\} = 2r. \] 
So for any $y \in Y_r$ there exists a point $z \in \partial_{\Omega} B_{\Omega}(y,2r/c_u)$ with $d_{\Omega}(z,\partial_{\widetilde{\Omega}} \Omega) \geq 2r$. Thus, 
$B(z,r) \subset B(y,A_7 r) \setminus Y_r$ if $A_7 = 2/c_u + 1$. The capacity of $B(y,A_7 r) \setminus Y_r$ in $B(y,2A_7 r)$ is larger than the capacity of $B(z,r)$ in 
$B(y,2A_7 r)$, which is larger than the capacity of $B(z,r)$ in $B(z,3A_7 r)$. The latter capacity is comparable to $V(z,r)/\Psi(r)$ by RES($\Psi$), which is satisfied by Theorem \ref{thm:VD+PI+CSA}. On the other hand, RES($\Psi$) implies that $\mbox{Cap}_{B(y,2A_7 r)}(B(y,A_7 r))$ is also comparable to $V(z,r)/\Psi(r)$. Hence, there exists $\eta>0$ so that $w_{\eta}(Y_r) \leq A_7 r$.
\end{proof}

Fix $\eta \in (0,1)$ small enough so that the conclusion of Lemma \ref{lem:4.12} applies and write $w(U) := w_{\eta}(U)$ for the capacitary width of an open set 
$U \subset \Omega$. 

The following lemma relates the capacitary width to the harmonic measure $\omega$. We write $f \asymp g$ to indicate that $C_1 g \leq f \leq C_2 g$, for some 
constants $C_1, C_2$ that only depend on $c_u$, $C_u$, $\beta_1$, $\beta_2$, $C_{\Psi}$ and the constants in VD and w-HKE($\Psi$).

For any open set $B \subset X$ let $\partial_X B := \overline{B} \setminus B$ denote the boundary of $B$ with respect to the metric $d$ on $X$. 

\begin{lemma} \label{lem:4.13}
There is a constant $a_1 \in (0,1]$ depending only on $\beta_1$, $\beta_2$, $C_{\Psi}$ and the constants in \emph{VD} and \emph{w-HKE($\Psi$)}, such that for any non-empty open set $U \subset X$ and any $x \in U$, $0 < r < \frac{1}{4} \mbox{\em diam}_d(X)$, we have
 \[ \omega(x,U \cap \partial_X B(x,r),U \cap B(x,r)) \leq \exp(2 - a_1 r / w(U)). \]
\end{lemma}

\begin{proof} The proof is nearly the same as in \cite[Lemma 1]{Aik01}, \cite[Lemma 4.13]{GyryaSC} and \cite[Lemma 4.8]{LierlSC1}. The comparison of capacities used in the proof follows easily from RES($\Psi$), which holds by Theorem \ref{thm:VD+PI+CSA}.
\end{proof}

By Lemma \ref{lem:x_r}, there exists, for any $\xi \in \partial_{\Omega} \Omega$, $r>0$, a point $\xi_{r}$ in $\Omega$ with $d_{\Omega}(\xi,\xi_{r}) = 4r$ and 
\[ d(\xi_{r}, X \setminus \Omega) \geq 2 c_u r. \]

\begin{lemma} \label{lem:4.14}
There exist constants $A_2, A_3 \in (0,\infty)$ depending only on $c_u$, $C_u$, $\beta_1$, $\beta_2$, $C_{\Psi}$ and the constants in \emph{VD} and \emph{w-HKE($\Psi$)}, such that for any $0 < r < R_{\xi}$ and any $x \in B_{\Omega}(\xi,r)$, we have
 \[  \omega \big(x, \partial_{\Omega} B_{\Omega}(\xi,2r),B_{\Omega}(\xi,2r) \big) 
     \leq  A_2 \frac{ V(\xi,r) }{ \Psi(r) }  G_{B_{\Omega}(\xi,A_3 r)} (x,\xi_{r}). \]
\end{lemma}

\begin{proof}
We follow \cite[Lemma 4.7]{LierlSC1}.
Let $A_3 = 2(12 + 12 C_u)$ so that all $(c_u,C_u)$-inner uniform paths that connect two points in $B_{\Omega}(\xi,12r)$ stay in $B_{\Omega}(\xi,A_3 r / 2)$.
Again we set  $A_0 = A_3+7$ and $Y := B_{\Omega}(\xi, A_0 r)$. 
For $z \in B_{\Omega}(\xi,A_3 r)$, set 
 \[  G'(z) :=  G_{B_{\Omega}(\xi,A_3 r)} (z,\xi_{r}). \]
Let $s =  \min\{ c_u r, 5r/C_u \}$. 
As $B_{\Omega}(\xi_{r},s)  \subset  B_{\Omega}(\xi,A_3r) \setminus B_{\Omega}(\xi,2r)$,
the maximum principle of Proposition \ref{prop:max principle} yields
 \[  \forall y \in B_{\Omega}(\xi,2r), \quad 
     G'(y)  \leq  \sup_{z \in \partial_{\Omega} B_{\Omega}(\xi_{r},s) }  G'(z). \]
Lemma \ref{lem:4.9} and the volume doubling condition yield
 \[  \sup_{z \in \partial_{\Omega} B_{\Omega}(\xi_{r},s) }  G'(z)   \leq  C \frac{\Psi(r)}{V(\xi,r)}, \]
for some constant $C>0$.
Hence, there exists $\epsilon_1 > 0$ such that
 \[ \forall y \in B_{\Omega}(\xi,2r), \quad 
    \epsilon_1  \frac{V(\xi,r)}{\Psi(r)} G'(y)  \leq e^{-1}. \] 
Write $ B_{\Omega}(\xi,2r) = \bigcup_{j \geq 0} U_j \cap B_{\Omega}(\xi,2r)$,
where
\[ U_j = \left\{ x \in Y : \exp (-2^{j+1}) \leq \epsilon_1  \frac{V(\xi,r)}{\Psi(r)} G'(x) < \exp(-2^j) \right\}. \]
Let $ V_j := \bigcup_{k \geq j} U_k$.
We claim that
\begin{equation} \label{eq:4.14} 
 w_{\eta}\big(V_j \cap B_{\Omega}(\xi, 2r) \big) \leq A_4 r \exp \left( - 2^j / \sigma \right)
\end{equation}
for some constants $A_4, \sigma \in (0,\infty)$.

Suppose $x \in V_j \cap B_{\Omega}(\xi,2r)$. Note that then $ d_{\Omega}(x, \partial_{\widetilde\Omega}\Omega) = d_{\Omega}(x, \widetilde{\Omega} \setminus Y)$. Observe that for $z \in \partial_{\Omega} B_{\Omega}(\xi_{r},s)$, by the inner uniformity of the domain, the length of the Harnack chain of balls in $B_{\Omega}(\xi, A_3 r) \setminus \{ \xi_{r} \}$ connecting $x$ to $z$ is at most $A_5 \log (1 + A_6 r / d_{\Omega}(x,\partial_{\widetilde\Omega}\Omega)$ for some constants $A_5, A_6 \in (0,\infty)$. Therefore there are constants $\epsilon_2, \epsilon_3, \sigma > 0$ such that
\begin{align*}
 \exp (-2^j) 
& > \epsilon_1 \frac{V(\xi,r)}{\Psi(r)} G'(x)
\geq \epsilon_2 \frac{V(\xi,r)}{\Psi(r)} G'(z) \left( \frac{ d_{\Omega}(x, \partial_{\widetilde\Omega}\Omega) }{ r } \right)^{\sigma} \\
& \geq \epsilon_3 \left( \frac{ d_{\Omega}(x, \partial_{\widetilde\Omega}\Omega) }{ r } \right)^{\sigma}.
\end{align*}
The last inequality is obtained by applying Lemma \ref{lem:4.9} with $R = A_3r$ and $\delta = 5/A_3$. Now we have that for any $x \in V_j \cap B_{\Omega}(\xi, 2r)$,
 \[ d_{\Omega}(x, \partial_{\widetilde\Omega}\Omega)
  \leq \big( \epsilon_3^{-1/\sigma} \exp(-2^j / \sigma) r \big) \wedge 2r. \]
This together with Lemma \ref{lem:4.12} yields \eqref{eq:4.14}.

Let $R_0 = 2r$ and for $j \geq 1$,
 \[ R_j = \left( 2 - \frac{6}{\pi^2} \sum_{k=1}^j \frac{1}{k^2} \right) r. \]
Then $R_j \downarrow r$ and
\begin{align} \label{eq:4.15}
\sum_{j=1}^{\infty}  \exp \left( 2^{j+1} - \frac{ a_1 (R_{j-1} - R_j) }{ A_4 r \exp(-2^j / \sigma) } \right)
& = \sum_{j=1}^{\infty}  \exp \left( 2^{j+1} - \frac{ 6 a_1 }{ A_4 \pi^2 } j^{-2} \exp(2^j / \sigma) \right) \nonumber \\
& \leq \sum_{j=1}^{\infty}  \exp \left( 2^{j+1} - \frac{ 3 a_1 }{ C_{\Omega} A_4 \pi^2 } j^{-2} \exp(2^j / \sigma) \right) \nonumber \\
& < C < \infty.
\end{align}
Let $\omega_0 = \omega( \cdot, \partial_{\Omega} B_{\Omega}(\xi, 2r), B_{\Omega}(\xi, 2r))$ and
 \[ d_j = \begin{cases} 
\sup \left\{ \frac{ \Psi(r) \omega_0(x) }{ V(\xi,r) G'(x) }  : x \in U_j \cap B_{\Omega}(\xi,R_j) \right\},  \quad 
    & \mbox{if } U_j \cap B_{\Omega}(\xi, R_j) \neq \emptyset, \\
0,  & \mbox{if } U_j \cap B_{\Omega}(\xi, R_j) = \emptyset.
\end{cases}
 \]
Since the sets $U_j \cap B_{\Omega}(\xi,2r)$ cover $B_{\Omega}(\xi,2r)$ and $B_{\Omega}(\xi,r) \subset B_{\Omega}(\xi,R_k)$ for each $k$, to prove Lemma \ref{lem:4.14}, it suffices to show that
 \[  \sup_{j \geq 0} d_j \leq A_2 < \infty \]
where the constant $A_2$ is as in the statement of Lemma \ref{lem:4.14}.

We proceed by iteration. Since $\omega_0 \leq 1$, we have by definition of $U_0$,
 \[  d_0 = \sup_{ U_0 \cap B_{\Omega}(\xi,2r) }  \frac{ \Psi(r) \omega_0(x) }{ V(\xi,r) G'(x) }  \leq  \epsilon_1 e^2. \]
Let $j \geq 1$. For $y \in U_{j-1} \cap B_{\Omega}(\xi, R_{j-1})$, we have by definition of $d_{j-1}$ that
\begin{align} \label{eq:U_{j-1}}
 \omega_0(y) \leq d_{j-1} \frac{  V(\xi,r) }{ \Psi(r) } G'(y).
\end{align}
On the other hand, we have $\omega_0(y) - d_{j-1} \frac{  V(\xi,r) }{ \Psi(r) } G'(y) \leq 1$ for $y \in V_j \cap \partial_X B(\xi,R_{j-1})$. 
Now let $x \in U_j \cap B_{\Omega}(\xi, R_j)$. Lemma \ref{lem:harm meas cts} and the mean value property \eqref{eq:mean value property} yield  
\begin{align} \label{eq:4.16}
 \omega_0(x)
& = d_{j-1} \frac{  V(\xi,r) }{ \Psi(r) } G'(x) \nonumber \\
& + 
\int_{\partial_X \big(V_j \cap B_{\Omega}(\xi,R_{j-1})\big)}
\omega_0(y) - d_{j-1} \frac{  V(\xi,r) }{ \Psi(r) } G'(y) \, \omega \left( x, dy , V_j \cap B_{\Omega}(\xi,R_{j-1}) \right) \nonumber \\
& \leq   d_{j-1} \frac{ V(\xi,r) }{ \Psi(r) } G'(x) 
\ + \ \omega \left( x,  V_j  \cap \partial_X B_{\Omega}(\xi,R_{j-1}), V_j \cap B_{\Omega}(\xi,R_{j-1}) \right).
\end{align}
Here we used \eqref{eq:U_{j-1}} and the fact that $(\partial_X V_j) \cap \Omega$ is a subset of $U_{j-1}$.

Let $D := B(x,C_{\Omega}^{-1}(R_{j-1}-R_j))$ and 
let $D'$ be the connected component of $p^{-1}(D \cap \overline{\Omega})$ that contains $x=p(x)$. By Lemma \ref{lem:metrics are comparable}, we have
 \[ D' \cap \Omega \subset B_{\Omega}(x,R_{j-1} - R_j) \subset B_{\Omega}(\xi,R_{j-1}). \]
Thus, by the monotonicity property \eqref{eq:monotonicity of harm meas} the second term on the right hand side of \eqref{eq:4.16} is not greater than
$ \omega \big( x,  V_j  \cap \partial_X (D' \cap \Omega), V_j \cap D' \cap \Omega \big)$.
Since $V_j \cap D' \cap \Omega = V_j \cap D' \cap D$ and
$V_j  \cap \partial_X (D' \cap \Omega) = V_j \cap D' \cap \partial_X D$,
we obtain that
 \begin{align*}
& \omega \big( x,  V_j  \cap \partial_X B_{\Omega}(\xi,R_{j-1}), V_j \cap B_{\Omega}(\xi,R_{j-1}) \big) \\
& \leq
 \omega \bigg(x, V_j \cap D' \cap \partial_{X} B\left(x,C_{\Omega}^{-1} (R_{j-1}-R_j) \right), V_j \cap D' \cap B\left(x,C_{\Omega}^{-1}(R_{j-1}-R_j)\right) \bigg) \\
& \leq
 \exp \left( 2 - a_1 \frac{ C_{\Omega}^{-1} (R_{j-1} - R_j )}{ w_{\eta}(V_j \cap D') } \right) 
 \leq \exp \left( 2 - \frac{a_1}{C_{\Omega}} \frac{ R_{j-1} - R_j }{ w_{\eta}(V_j \cap B_{\Omega}(\xi, 2r)) } \right) \\
& \leq 
\exp \left( 2 - \frac{a_1}{C_{\Omega} A_4} \exp( 2^j / \sigma) \frac{ R_{j-1} - R_j }{ r } \right)  
 \leq \exp \left( 2 - \epsilon_6 j^{-2} \exp( 2^j / \sigma) \right),
\end{align*}
by Lemma \ref{lem:4.13}, the monotonicity of $U \mapsto w_{\eta}(U)$ and \eqref{eq:4.14}. Here $\epsilon_6 = \frac{6 a_1}{\pi^2 A_4 C_{\Omega}}$. 
Moreover, since $x \in U_j \cap B_{\Omega}(\xi,R_j)$, we have by definition of $U_j$ that
 \[ \epsilon_1 \frac{V(\xi,r)}{\Psi(r)} G'(x) \geq \exp(-2^{j+1}). \]
Hence, inequality \eqref{eq:4.16} becomes
\begin{align*}
\omega_0(x)  & \leq  \exp \left( 2 - \epsilon_6 j^{-2} \exp( 2^j / \sigma) \right) 
                   +  d_{j-1} \frac{ V(\xi,r) }{ \Psi(r) } G'(x)  \\
             & \leq  \left( \epsilon_1 \exp \left( 2 + 2^{j+1} - \epsilon_6 j^{-2} \exp( 2^j / \sigma) \right) 
                   +  d_{j-1} \right) \frac{ V(\xi,r) }{ \Psi(r) } G'(x).
\end{align*}
Dividing both sides by $\frac{ V(\xi,r) }{ \Psi(r) } G'(x)$ and taking the supremum over all
$x \in U_j \cap B_{\Omega}(\xi,R_j)$,
 \[ d_j  \leq   \epsilon_1 \exp \left( 2 + 2^{j+1} - \epsilon_6 j^{-2} \exp( 2^j / \sigma) \right) 
                   +  d_{j-1}, \]
and hence for every integer $i > 0$,
\begin{align*}
 d_i  
 & \leq  \epsilon_1 e^2 \bigg( 1 + \sum_{j=1}^{i} \exp  \left(2^{j+1} - \epsilon_6 j^{-2} \exp( 2^j / \sigma) \right) \bigg) 
 = \epsilon_1 e^2 (1 + C) < \infty
\end{align*}
by \eqref{eq:4.15}.
\end{proof}

\begin{proof}[Proof of Theorem \ref{thm:bHP with GF's}]
In view of the results of this section, the proof of \cite[Theorem 4.1]{LierlSC1} carries over line by line. The only difference is that, as in the Green function estimates of Lemma \ref{lem:4.8} and Lemma \ref{lem:4.9}, the space-time scaling is $\Psi(r)$ instead of $r^2$.
\end{proof}

\section{Applications} \label{sec:applications}

\subsection{Martin boundary}
\begin{theorem}
Let $(X,d,\mu,\e,\F)$ be a \emph{MMD} space that satisfies \emph{VD} and \emph{w-HKE($\Psi$)}  and $d$ is geodesic.
Let $\Omega \subset X$ be a bounded inner uniform domain in $X$. Then the Martin compactification relative to $(\e,\F)$ of $\Omega$ is homeomorphic to $\widetilde{\Omega}$ and each boundary point $\xi \in \widetilde{\Omega} \setminus \Omega$ is minimal.
\end{theorem}

\begin{proof}
The assertion can be proved along the line of \cite[Theorem 1.1]{ALM03} using the boundary Harnack principle of Theorem \ref{thm:bHP for u}.
\end{proof}

\subsection{Existence of a harmonic profile}

On inner uniform domains, the geometric boundary Harnack principle is crucial to obtain precise estimates for the Dirichlet heat kernel. See \cite{GyryaSC, LierlSC3}, which cover the case of regular Dirichlet spaces that induce a non-degenerate metric and satisfy a parabolic Harnack inequality.

In the case of unbounded domains, an important technique to obtain Dirichlet heat kernel estimates is the $h$-transform of the Dirichlet form by a \emph{harmonic profile} $h$ for $(\e,\F)$ on the domain $U$. 
In this section, we prove the existence of a harmonic profile on inner uniform domains in a MMD space. 

In forthcoming papers \cite{KL1, KL2}, the harmonic profile, as well as the geometric boundary Harnack principle of Theorem \ref{thm:bHP for u}, will be used to prove two-sided bounds for the Dirichlet heat kernel in inner uniform domains on fractal spaces.

\begin{definition}
A function $h$ on an unbounded domain $U$ is called a \emph{harmonic profile} if
\begin{enumerate}
\item
$h$ is harmonic on $U$,
\item
$h$ satisfies Dirichlet boundary condition along $\partial_{\widetilde{U}} U$, that is, $h \in \F_{\mbox{\tiny{loc}}}^0(U)$,
\item
$h>0$ in $U$.
\end{enumerate}
\end{definition}

\begin{remark}
Due to the maximum principle, a harmonic profile can exist only if the domain is unbounded. On bounded domains, Dirichlet heat kernel estimates can be obtained by using a ground state transform instead of an $h$-transform.
\end{remark}

\begin{proposition}
Let $(X,d,\mu,\e,\F)$ be a \emph{MMD} space that satisfies \emph{VD} and \emph{w-HKE($\Psi$)}  and $d$ is geodesic. Let $U$ be an unbounded inner uniform domain in $X$. Then there exists a harmonic profile for $(\e,\F)$ on $U$. It is unique up to multiplication by a scalar.
\end{proposition}

\begin{proof}
Uniqueness can be deduced from the boundary Harnack principle and a classical argument, see, e.g., \cite[Section 6]{Anc78}.
Existence can be proved as in the non-fractal case, see\cite[Section 4.3]{GyryaSC}, except that CSA($\Psi$) must be used instead of cutoff functions constructed from the metric induced by $(\e,\F)$.
\end{proof}

\section*{Acknowledgement}
I thank Naotaka Kajino for many helpful discussions and for his comments on an earlier version of this paper. I  thank Laurent Saloff-Coste for answering questions. I thank the referees for numerous useful comments.

\bibliographystyle{alpha}
\bibliography{bibliography}
\end{document}